\documentclass[11pt,leqno]{amsart}
\usepackage[margin=3.8cm]{geometry}
\usepackage[utf8]{inputenc}
\usepackage{graphicx, array, url, tcolorbox, xcolor, hyperref, amsmath, amsfonts, amsthm, amssymb, amstext, comment, latexsym}
\usepackage{tikz}
\usepackage{float, comment}
\usepackage{caption}
\usetikzlibrary{arrows.meta}
\usepackage[normalem]{ulem}
\usepackage[active]{srcltx}
\usepackage[colorinlistoftodos,prependcaption]{todonotes}

\newcommand{\bfx}{\ensuremath{\mathbf{x}}}
\newcommand{\bfy}{\ensuremath{\mathbf{y}}}

\newcommand{\Adm}{\operatorname{Adm}}
\newcommand{\Val}{\operatorname{Val}}
\newcommand{\val}{\operatorname{val}}
\newcommand{\dist}{\operatorname{dist}}

\newcommand{\N}{\ensuremath{\mathbb{N}}}
\newcommand{\R}{\ensuremath{\mathbb{R}}}

\newcommand{\Q}{\ensuremath{\mathbb{Q}}}

\newcommand{\diam}{\ensuremath{\ \mathrm{diam}}}
\newcommand{\Lip}{\ensuremath{\mathrm{Lip}}}

\newcommand{\bfc}{\textbf{c}}
\newcommand{\bfn}{\textbf{n}}
\newcommand{\ladim}{\dim_{\mathrm{LA}}}
\newcommand{\adim}{\dim_{\mathrm{A}}}

\newtheorem{question}{Question}
\newtheorem{theorem}{Theorem}[section]
\newtheorem{lemma}[theorem]{Lemma}
\newtheorem{proposition}[theorem]{Proposition}
\newtheorem{definition}[theorem]{Definition}
\newtheorem{corollary}[theorem]{Corollary}
\newtheorem{remark}[theorem]{Remark}
\newtheorem{notation}[theorem]{Notation}
\definecolor{peach}{HTML}{F7965A}
\definecolor{springgreen}{HTML}{C6DC67}
\definecolor{royalblue}{HTML}{0071BC}
\definecolor{periwinkle}{HTML}{7977B8}
\definecolor{forestgreen}{HTML}{009B55}

\title{{On universal elements for doubling geodesic trees}}
\author{Sylvester Eriksson-Bique}
\address{Department of Mathematics and Statistics
P.O. Box 35
FI-40014 University of Jyväskylä}
\email{sylvester.d.eriksson-bique@jyu.fi}
\author{Manisha Garg}
\address{Department of Mathematics, University of Illinois at Urbana-Champaign, Urbana, Illinois, USA}
\email{manisha8@illinois.edu, manishagmath@gmail.com}

\date{\today}

\begin{document}

\maketitle

\begin{abstract}
For $n\ge 3$ and $c\in(0,1)$, let $\mathcal{GT}(n,c)$ denote the class of geodesic metric trees of valence at most $n$ whose branch points are uniformly relatively separated with constant $c$. We prove that $\mathcal{GT}(n,c)$ has no bi-Lipschitz universal element. More precisely, we construct a family $(T_a)_{a\in[1/4,1/3]}\subset\mathcal{GT}(n,c)$ such that, for every $n_M\geq 3, c_M\in(0,1)$ and every $M\in\mathcal{GT}(n_M, c_M)$, there are at most countably many parameters $a$ for which $T_a$ admits a bi-Lipschitz embedding into $M$, whereas each $T_a$ admits a bi-Lipschitz embedding into $\R^2$. Thus the obstruction is neither dimensional nor caused by a failure of planar embeddability. This gives a negative answer to a question of Chrontsios-Garitsis, Ioannidis, and Vellis~\cite[Question~1.11]{CGIV2024}. Furthermore, we show a complementary positive result for ultrametric spaces: every bounded ultrametric space $X$ admits a bi-Lipschitz embedding into every complete metric space $Y$ satisfying $\adim X<\ladim Y$ where $\adim X$ and $\ladim X$ are Assouad and lower Assouad dimensions, respectively.
\end{abstract}

\tableofcontents

\section{Introduction}
An important question in analysis and geometry is whether a given metric space is equivalent to or naturally contained in a model space with a well understood geometry. In particular, the uniformization problem asks whether a given metric space can be replaced, within a specified category of maps, by a geometrically simpler or more canonical representative.  On the other hand, the embedding problem asks instead whether the space can be realized, with controlled distortion, as a subset of a fixed and well-understood ambient space.   Quasisymmetric uniformization has been particularly effective for arcs, metric surfaces, and planar fractals \cite{TV80,BK02,Bon11}, and it also plays an important role in the study of boundaries of hyperbolic groups \cite{BK05}.  In each problem the choice of maps is essential. Quasisymmetric maps control relative distances and form the natural category in the quasiconformal geometry, whereas bi-Lipschitz maps control all distances by one multiplicative constant and are therefore substantially more rigid; every bi-Lipschitz embedding is quasisymmetric, but not conversely.  For general background, see \cite{HeinonenBookAnalysis}.

The spaces considered in this paper are metric trees and the geometric maps considered are bi-Lispchitz maps.  A metric tree is a compact, connected, and locally connected metric space $T$ in which every two points $x,y\in T$ are joined by a unique arc, denoted by $[x,y]$.  A metric tree is called a \emph{quasiconformal tree} if it is doubling and of bounded turning (see \S~\ref{S:prelim} for definitions). This class contains all quasiarcs and all doubling geodesic trees. Kinneberg considered linearly connected metric trees under the name \emph{quasi-trees} and proved, in particular, that their doubling members have conformal dimension one \cite{Kinneberg:confdim17}.  The term \emph{quasiconformal tree} was subsequently introduced by Bonk and Meyer, who developed the quasisymmetric geometry of the class \cite{BonkMeyer:qcgeodtrees20,BonkMeyer:uniqctrees22}. Quasiconformal trees also arise naturally in planar analysis and complex dynamics.  For example, planar bounded-turning trees are closely related to complements that are John domains \cite[Theorem~4.5]{NV91}, while tree-like Julia sets of suitable semihyperbolic polynomials provide dynamical examples \cite{CG93,CJY94}.  Other standard examples include the continuum self-similar tree and the Vicsek fractal; see \cite{BonkTran:csst21,BonkMeyer:uniqctrees22}. Moreover, in geometric group theory, quasi-arcs have been utilized to study quasi-isometric embedding of $\mathbb H^2$ in hyperbolic groups \cite{BK05:quasiplanes, HruskaRuane2025:quasiplanes}.

For quasiconformal trees, the uniformization problem has a strong positive answer.  The non-branching prototype is the theorem of Tukia and
V\"ais\"al\"a: a metric arc is quasisymmetrically equivalent to an interval if and only if it is doubling and of bounded turning \cite{TV80}.  Bonk and Meyer extended this result from arcs to trees by proving that every quasiconformal tree is quasisymmetrically equivalent to a geodesic tree \cite{BonkMeyer:qcgeodtrees20}.  More precisely, for every $s>1$, the geodesic representative may be chosen to have
Hausdorff dimension at most $s$.  Geodesic trees form a particularly natural model class. If $T$ is a geodesic tree, then the unique arc $[x,y]$ is isometric to $[0,d(x,y)]$. Consequently, $\diam[x,y] = d(x,y)$. Note that the resulting quasisymmetric equivalence need not be bi-Lipschitz.

% Thus all arcs in a quasiconformal tree can be straightened simultaneously, although the resulting quasisymmetric equivalence need not be bi-Lipschitz.
% Under stronger quantitative assumptions on the branching, Bonk and Meyer also characterized the quasiconformal trees that are quasisymmetrically equivalent to the continuum self-similar tree \cite{BonkMeyer:uniqctrees22}.

Uniformization provides a geodesic model separately for each quasiconformal tree. A stronger model-space problem asks whether one can find a single space containing `controlled' copies of every member of a prescribed class of quasiconformal trees. This single space could be a Euclidean space or an element within this class. More precisely, given a class $\mathcal C$ of metric spaces, one seeks a single space $U$ called a universal element, preferably belonging to $\mathcal C$, into which every member of $\mathcal C$ embeds with the prescribed type of control. 

Topologically, the class of metric trees has universal elements \cite{Nad92}.  In the quasisymmetric category, Chrontsios-Garitsis, Ioannidis, and Vellis recently obtained a quantitative counterpart.  For each $n\ge 3$, they constructed a geodesic quasiconformal tree $\mathbb T_n$ that is quasisymmetrically universal for quasiconformal trees of valence at most $n$ whose branch points are uniformly relatively separated \cite{CGIV2024}.  They also showed that every quasiconformal tree with this branch-separation property admits a quasisymmetric embedding into $\R^2$ with quasiconvex image.  Thus both uniformization and universality are positive in this quasisymmetric setting.

The corresponding bi-Lipschitz embedding theory is more rigid.  Doubling is
necessary for a metric space to embed bi-Lipschitzly into a
finite-dimensional Euclidean space, but it is not sufficient in general.
Assouad's embedding theorem gives a fundamental partial substitute: if
$(X,d)$ is doubling and $0<\alpha<1$, then the snowflake
$(X,d^\alpha)$ admits a bi-Lipschitz embedding into some Euclidean space
\cite{Ass83}. This means every doubling metric space admits a quasisymmetric embedding into a Euclidean space, but Assouad's theorem does not generally give a bi-Lipschitz embedding of the original metric. Low-distortion embeddings of this kind are also central in theoretical computer science, where one seeks to represent a complicated metric inside a more tractable Euclidean or tree-like space while approximately preserving all pairwise distances \cite{book:matousekdiscgeom, NR03, Ind01, Lin02}.

In the class of doubling geodesic trees, Gupta, Krauthgamer, and Lee proved that every
tree embeds bi-Lipschitzly into a finite-dimensional Euclidean space \cite{GKL03}; see also the different approach of Lee, Naor, and Peres \cite{LNP09}. David and Vellis subsequently showed that a quasiconformal tree embeds bi-Lipschitzly into some Euclidean space if and only if its set of leaves does so \cite{DV22:qctrees}. Finally, David, Eriksson-Bique, and Vellis proved that every quasiconformal tree admits a bi-Lipschitz embedding into some $\R^N$, where both $N$ and the bi-Lipschitz constant depend only on the doubling and bounded-turning constants of the tree \cite{DEBV23:qctrees}.

These theorems give external Euclidean models to embed quasiconformal trees, but they do not produce an internal universal model that is itself a tree.
To formulate the latter problem, let $n\ge 3$ and $c\in(0,1)$, and let
$\mathcal{GT}(n,c)$ denote the class of geodesic metric trees with valence at
most $n$ and uniformly relatively separated branch points with constant $c$;
see Section~\ref{S:prelim}.  
A member $U\in\mathcal{GT}(n,c)$ is called
\emph{bi-Lipschitz universal} for $\mathcal{GT}(n,c)$ if every tree in this
class admits a bi-Lipschitz embedding into $U$, with no requirement that the
distortion be uniform over the class.  
The assumptions defining $\mathcal{GT}(n,c)$ imply a doubling bound depending only on $n$ and $c$ \cite[Lemma~2.5]{CGIV2024}.  
Since every geodesic tree is $1$-bounded turning,
the theorem of David, Eriksson-Bique, and Vellis therefore gives, for fixed
$n$ and $c$, a dimension $N=N(n,c)$ such that every member of
$\mathcal{GT}(n,c)$ embeds bi-Lipschitzly into the same Euclidean space
$\R^N$.  Chrontsios-Garitsis, Ioannidis, and Vellis asked whether the
Euclidean target can be replaced by a universal target inside the class:

\begin{question}[{\cite[Question~1.11]{CGIV2024}}]\label{ques}
Does there exist a tree $T_0\in\mathcal{GT}(n,c)$ into which every member of
$\mathcal{GT}(n,c)$ embeds bi-Lipschitzly?  If such a tree exists, is it
bi-Lipschitz homeomorphic to $\mathbb T_n$?
\end{question}

We answer this question negatively. More precisely, for each fixed $c \in (0,1)$ we construct an uncountable one-parameter family $\{T_{a, c} \mid a \in [1/4,1/3]\}$ of trivalent geodesic trees such that $T_a \in \mathcal{GT}(3, c)\subseteq \mathcal{GT}(n,c)$ for every $a\in [1/4,1/3]$ and  $n\ge 3$.  Our main theorem is the following stronger countability obstruction.

\begin{theorem}\label{thm:mainthm1}
Let $n\ge3$, let $c\in(0,1)$, and let
$M\in\mathcal{GT}(n_M,c_M)$ for $n_M$ and $c_M$ possibly distinct from $n$ and $c$, respectively. Then the set
\[
    \left\{
    a\in[1/4,1/3]:
    T_{a,c}\text{ admits a bi-Lipschitz embedding into }M
    \right\} 
\]
is at most countable. For each $a\in[1/4,1/3]$, $T_{a,c}\mathcal{QC}(n,c)$.
\end{theorem}

% \begin{theorem}\label{thm:mainthm1}
% Let $n\ge 3$, let $c\in(0,1)$, and let
% $M\in\mathcal{GT}(n,c)$.  Then there are at most countably many parameters
% $a\in[1/4,1/3]$ for which $T_a$ admits a bi-Lipschitz embedding into $M$.
% \end{theorem}

Indeed, fix $n\ge3$ and $c\in(0,1)$. If $U\in\mathcal{GT}(n,c)$ were bi-Lipschitz universal, then every tree $T_{a,c}$ would admit a bi-Lipschitz embedding into $U$.
This contradicts Theorem~\ref{thm:mainthm1}, since $[1/4,1/3]$ is uncountable.

\begin{corollary}\label{cor:no_universal_tree}
For every $n\ge 3$ and $c\in(0,1)$, the class
$\mathcal{GT}(n,c)$ has no bi-Lipschitz universal element. In fact, the class
$$
    \bigcup_{\substack{n\geq3\\ c\in(0,1)}}
    \mathcal{GT}(n,c)$$
has no bi-Lipschitz universal element.
\end{corollary}

% \begin{proof}
% Suppose that $U$ were a bi-Lipschitz universal element for the displayed
% union. Since $U$ belongs to this union, there exist $n_0\geq3$ and
% $c_0\in(0,1)$ such that $U\in\mathcal{GT}(n_0,c_0)$. Universality for the union would imply that every member of $\mathcal{GT}(n_0,c_0)$ embeds bi-Lipschitzly into $U$. Thus $U$ would be a bi-Lipschitz universal element for $\mathcal{GT}(n_0,c_0)$, contradicting Corollary~\ref{cor:no_universal_tree}.
% \end{proof}

The obstruction is neither topological nor dimensional.  The trees $T_a$ have
the same combinatorial construction, and all of them satisfy
\[
    \dim_H(T_a)=\dim_N(T_a)=\adim (T_a)=1.
\]
We note that the Assouad dimension is determined by a general principle involving uniform bounds over Whitney scales, which seems to be of use in other settings as well, see Lemma \ref{lem:whitney_scale_assouad} for details. 

Every $T_a$ admits a bi-Lipschitz embedding into $\R^2$.  Thus every individual example has an especially simple Euclidean model, while no single tree in
$\mathcal{GT}(n,c)$ can contain bi-Lipschitz copies of the entire family.  The
parameter $a$ is detected instead by the relative lengths of branches through
successive generations, information that is invisible to topology and to the
standard notions of dimension but rigid under bi-Lipschitz embeddings.

%The iterated subdivision used to construct $T_a$ also determines a rooted tree of intervals whose space of infinite branches carries a
%natural ultrametric. 
Trees are the simplest possible connected metric spaces. We also consider the
corresponding universality problem for ultrametric spaces and obtain a
positive result discussed below. This is made possible by the fact that such spaces are totally disconnected.

\subsection{A positive comparison: ultrametric spaces}

Theorem~\ref{thm:mainthm1} concerns connected geodesic trees and to clarify the role played by connectedness and the geometry of the connecting arcs, we also consider the corresponding embedding problem for ultrametric spaces. An ultrametric space is encoded by a rooted tree of nested balls, and the distance between two points is determined by the first level at which their branches separate. Unlike a geodesic tree, however, an ultrametric space contains no connecting arcs whose lengths must be preserved consistently.

In this setting, for an ultrametric space $X$, picking a target which has availability of branching at higher rate than $X$ gives a positive embedding theorem.

\begin{theorem}
\label{thm:um_embedding_intro}
Let $X$ be a bounded ultrametric space and let $Y$ be a complete
metric space. If
\[
    \adim X<\ladim Y,
\]
where $\adim X$ is the Assouad dimension of $X$ and $\ladim$ is the lower Assouad dimension of $Y$, 
then $X$ admits a bi-Lipschitz embedding into $Y$.
\end{theorem}
In particular, every fixed complete $Q$-Ahlfors regular metric space
is a bi-Lipschitz universal target for bounded ultrametric spaces of
Assouad dimension  $\, < \,Q$. The proof recursively
realizes the ball tree of $X$ by a separated family of nested balls
in $Y$: the Assouad estimate bounds the number of children of each
source ball, while the lower Assouad estimate supplies sufficiently
many separated locations for them in the target. 

This provides a positive comparison with
Corollary~\ref{cor:no_universal_tree}. It shows that hierarchical
branching alone does not explain the failure of bi-Lipschitz
universality for geodesic trees. The obstruction in
Theorem~\ref{thm:mainthm1} comes from the additional requirement that
the geometry of the connecting arcs remain compatible through
arbitrarily many generations.

There are several very closely related results on embeddings of ultrametric spaces, and our proof here follows known strategies. See Section \ref{S:ultrametric} for more discussion on the  background and related work.

\subsection{Proof strategy and main tools}

The form of the proof is dictated by the way in which the parameter $a$ is encoded in the trees $T_a$. At any fixed number of generations, the subdivision structures of $T_a$ and $T_b$ are quantitatively close whenever $a$ and $b$ are close. Consequently, a comparison at only finitely many scales cannot distinguish the two trees up to bi-Lipschitz equivalence: any finite discrepancy may be absorbed into the bi-Lipschitz constant. The distinction between the parameters becomes visible only after passing through arbitrarily many generations. Indeed, the typical length of a level-$k$ descendant in $T_a$ is of exponential order
$$
\exp(k\alpha_a),
\qquad
\alpha_a=2a\log a+(1-2a)\log(1-2a),
$$
and the strict monotonicity of $a\mapsto\alpha_a$ implies that the typical lengths associated with two distinct parameters eventually become exponentially incomparable.

To exploit this difference, however, it is necessary to compare descendants in $T_a$ and $T_b$ whose images determine the same arc in the target tree. An arbitrary pair of bi-Lipschitz embeddings need not preserve corresponding subdivisions, so this alignment cannot be assumed. We therefore first use countability and pigeonholing to choose two nearby parameters $a\neq b$, embeddings with common quantitative bounds, and suitable horizontal intervals having the same endpoint images in the target.

Moreover, we require the near maximal-stretching of segment by bi-Lipschitz map to propagate this initial agreement of endpoints. The key ingredient here is a one-dimensional near-maximal-stretching principle. If a Lipschitz map nearly realizes its Lipschitz constant on an interval, then only a small proportion of its descendant intervals can exhibit a definite loss of stretch. This is a one-dimensional analogue of the regular-square argument of Burago and Kleiner~\cite{buragokleiner}: near-saturation of the global Lipschitz bound leaves only a small ``defect budget'' for disjoint subintervals on which the map stretches substantially less. A related extremal viewpoint appears in Preiss's work on differentiability of Lipschitz maps in Banach spaces~\cite{Preissdiff}, where one considers points and directions along which a Lipschitz map almost realizes its maximal stretching.

Finally, this synchronization is needed for all deep enough levels and is available only along descendants that remain well-stretched with typical length as described above. The measure-theoretic ingredients ensure that both conditions are abundant.  Our selection of good intervals across many refinement levels as we proceed deeper in the iterations is similar in spirit to density and pigeonhole arguments appearing in work of Semmes~\cite{semmes}. 

% \begin{theorem}\label{thm:universal_ultrametric}
% Let $X$ be a bounded ultrametric space and let $Y$ be a $Q'$-Ahlfors regular ultrametric space, then $X$ embeds bi-Lipschitzly in $Y$.    
% \end{theorem}

\subsection{Structure of the paper} 
Section~\ref{S:prelim} contains the necessary background and preliminary results from the literature. We also prove a Whitney-scale covering lemma for Assouad dimension, which allows us to control the Assouad dimension of a metric space by combining estimates near a closed subset with estimates away from it. In Section~\ref{S:intervals_lip_maps}, we study Lipschitz maps on iteratively subdivided intervals. The main result of this section is a maximal-stretching principle which produces well-stretched subintervals at arbitrarily fine scales. In Section~\ref{S:T_a}, we construct the family of geodesic trees $(T_a)_{a\in[1/4,1/3]}$ and prove their basic geometric and metric properties. In Section~\ref{S:mainthmproof}, we prove Theorem~\ref{thm:mainthm1} and in Section~\ref{S:ultrametric} we prove bi-Lipschitz embeddability results for ultrametric spaces.

% The proof combines the maximal-stretching results from Section~\ref{S:intervals_lip_maps} with the tree geometry of the target space to show that a fixed tree $M\in\mathcal{GT}(n,c)$ can contain bi-Lipschitz copies of at most countably many of the trees $T_a$.

\subsection*{Acknowledgments} 
The authors gratefully acknowledges the support and hospitality of the Hausdorff Institute for Mathematics (HIM), Bonn, where the project was started during the Trimester program on Metric Analysis. This work was funded by the Deutsche Forschungsgemeinschaft (DfG, German Research Foundation) under the German Excellence Strategy – EXC-2047/1 – 390685813. We also thank IMPAN for support and hospitality during the Simons semester on Geometric analysis. This work was partially supported by the Simons Foundation grant (award no. SFI-MPS-T-Institutes-00010825) and from State Treasury funds as part of a task commissioned by the Minister of Science and Higher Education under the project “Organization of the Simons Semesters at the Banach Center - New Energies in 2026-2028” (agreement no. MNiSW/2025/DAP/491).

The first author is partially supported by the Research Council of Finland via the project GeoQuantAM: Geometric and Quantitative Analysis on Metric spaces, grant no. 354241. The second author was partially supported by the McNamara Education Grant by World Bank.

The authors are grateful to Efstathios-Konstantinos Chrontsios-Garitsis, Jeremy Tyson  and Vyron Vellis for many helpful discussions.

\section{Preliminaries and background}\label{S:prelim}
In this section, we collect the terminology and elementary facts used in the paper.
We first fix notation for Lipschitz maps and Hausdorff measure, and then
record the tree-geometric consequences of a bi-Lipschitz embedding. We end
with the dimension notions used to describe the examples and a
Whitney-scale lemma for Assouad dimension.

\subsection{Notation and mappings}
Throughout the paper, $\bfn \ge 3$ and $\bfc\in(0,1)$ are fixed. We write $\mathcal{GT}(\bfn, \bfc)$ for the class of geodesic metric trees with valence at most $\bfn$ and uniformly relatively separated branch points with constant $\bfc$; to avoid
burdening the notation, we write $T_a$ instead of $T_{a,\bfc}$.  When the
parameters of a target tree need not agree with those of the source family,
we denote them by $\bfn_M$ and $\bfc_M$.

For any closed interval $J \subset \R$, write $J = [s_J, t_J]$, where $s_J < t_J$ and denote its length by $$|J|:=t_J-s_J.$$ 
We denote Lebesgue measure on $\R$ by $\mathcal L$ and the Hausdorff measure on $\R$ by $\mathcal H$ and both measures coincide in $\R$.

 A mapping $f: X\to Y$ between metric spaces is called \emph{Lipschitz} if there is a constant $L>0$ such that
$d(f (x), f (y))\, \le\, Ld(x, y)$ for all $x, y \in X$. The smallest such $L$ is called the Lipschitz constant, and it is denoted by $\Lip(f)$.

A mapping $f: X\to Y$ between metric spaces is called called \emph{$(1/k, K)-$bi-Lipschitz} if there are constants $k, K > 0$ such that $$1/k \,d(x, y) \le d(f (x), f (y) \le K\,d(x, y) \text{ for all } x, y \in X.$$ The smallest $K$ and the smallest $k$ satisfying this are called the Lipschitz and lower-Lipschitz constants for $f$. The pair $(k, K)$ will be referred to as the bi-Lipschitz constants of $f$. 

We will use the following standard distortion estimate for Hausdorff
measure. For $s\ge 0$ and $\delta>0$, let $\mathcal H^s_\delta$
denote the $s$-dimensional Hausdorff content obtained using covers by sets
of diameter at most $\delta$.

\begin{lemma}\label{lem:hmeausre_lispchitz_maps}
 Let $(X,d_X)$ and $(Y,d_Y)$ be metric spaces, let $s\ge 0$, and let $f:X\to Y$ be
$L$-Lipschitz. Then for every $E\subset X$ and every $\delta>0$,
\[
\mathcal H^{s}_{L\delta}\bigl(f(E)\bigr)\ \le\ L^{s}\,\mathcal H^{s}_{\delta}(E).
\]
In particular, $
\mathcal H^{s}\bigl(f(E)\bigr)\ \le\ L^{s}\,\mathcal H^{s}(E).$
\end{lemma}

\begin{proof}
If $\{U_i\}$ is a countable cover of $E$ with $\diam \,U_i\le \delta$, then
$\{f(U_i)\}$ is a cover of $f(E)$ and
\[
    \diam_Y\bigl(f(U_i)\bigr)\ \le\  L\delta
    \qquad\text{for all }i.
\]

Hence $\{f(U_i)\}$ is an admissible cover, and therefore
\[
    \mathcal H^{s}_{L\delta}\bigl(f(E)\bigr)
    \le
    \sum_{i} \bigl(\diam_Y(f(U_i))\bigr)^s
    \le
    L^{s}\sum_{i} \bigl(\diam_X(U_i)\bigr)^s.
\]
Taking the infimum over all such $\delta$-covers $\{U_i\}$ of $E$ yields the first claim. Letting $\delta \to 0$ yields the second claim. 
\end{proof}

\subsection{Metric trees and branch separation}

A \emph{metric tree} is a compact, connected, and locally connected metric
space containing no simple closed curve. Equivalently, every two distinct
points $x,y\in T$ are the endpoints of a unique arc, which we denote by
$[x,y]_T$. A metric tree $T$ is \emph{geodesic} if each arc
$[x,y]_T$ is isometric to the interval $[0,d_T(x,y)]$. In particular, $\diam[x,y]_T =d_T(x,y)$.
If $u,v\in[x,z]_T$, then $d_T(u,v)=\bigl|d_T(x,u)-d_T(x,v)\bigr|$.

For $p\in T$, the connected components of $T\setminus\{p\}$ are called
the \emph{branches of $T$ at $p$}. We denote them by
\[
    B_1^T(p),B_2^T(p),\ldots
\]
and arrange them so that their diameters are nonincreasing. The valence of
$p$, denoted by $\val_T(p)$, is the number of these
components, and $\Val(T):=\sup_{p\in T}\val_T(p)$. A point $p$ is a \emph{branch point} if
$\val_T(p)\ge3$. Its height is
\[
    H_T(p):=\diam B_3^T(p).
\]
We write $\mathcal B(T)$ for the set of branch points of $T$.

We say that $T$ has \emph{uniformly relatively separated branch points}
with constant $c\in(0,1]$ if
\[
    d_T(p,q)
    \ge c\min\{H_T(p),H_T(q)\}
    \qquad\text{whenever }p,q\in\mathcal B(T),\ p\neq q.
\]
For $n\ge2$ and $c\in(0,1]$, let $\mathcal{GT}(n,c)$ be the class of
geodesic metric trees $T$ such that $\Val(T)\le n$ and the preceding separation inequality holds with constant $c$. Notice that if $0<c\le c_0$, then
\[
    \mathcal{GT}(n,c_0)\subseteq\mathcal{GT}(n,c).
\]

A metric space is \emph{doubling} if some $D\ge1$ has the property that
every ball is covered by at most $D$ balls of half its radius. A metric
tree is \emph{bounded turning} if some $C\ge1$ satisfies $\diam[x,y]_T\le C d_T(x,y)$ for all $x,y\in T$.

\begin{definition}[Quasiconformal tree]
A doubling, bounded-turning metric tree is called a \emph{quasiconformal
tree}. Every geodesic tree is $1$-bounded turning and a doubling geodesic tree is a quasiconformal tree.  
\end{definition}

Uniform branch separation also gives a quantitative doubling bound.  
\begin{remark}
We state the result using the convention $c\in(0,1]$ above.  The convention
in \cite{CGIV2024} uses a constant $C\ge1$ and writes $C^{-1}$ in the
separation inequality, so the two constants are related by $C=c^{-1}$. 
\end{remark}

\begin{lemma}[{\cite[Lemma~2.5]{CGIV2024}}]
\label{lem:doubling_from_branch_sep}
For every $n\ge2$ and $c\in(0,1]$, there exists $D=D(n,c)$ such that
every tree in $\mathcal{GT}(n,c)$ is $D$-doubling.  One may take
\[
    D(n,c)=3(n-2)\left(\frac{6}{c}+1\right)+7.
\]
\end{lemma}

Next, we first record the countability fact noted in \cite[Chapter V, (1.3)(iv)]{Why65} and is used in the pigeonholing step of the main proof.

\begin{lemma}\label{lem:countable_branch_points}
The set of branch points of a quasiconformal tree is countable.
\end{lemma}

The following elementary observation will be used repeatedly in the
synchronization argument.

\begin{lemma}\label{lem:tree_embedding_arcs_heights}
Let $T$ and $M$ be metric trees, and let $f:T\to M$ be an embedding.
Then $ f([x,y]_T)=[f(x),f(y)]_M$ for all $x,y\in T$. In particular, $f(\mathcal B(T))\subseteq\mathcal B(M)$. If, in addition,
$f$ is $(1/K,L)$-bi-Lipschitz, then $ H_M(f(p))\ge 1/K H_T(p)$ for every $p\in\mathcal B(T)$.
\end{lemma}

\begin{proof}
The restriction of $f$ to $[x,y]_T$ is a homeomorphism onto an arc in
$M$ joining $f(x)$ to $f(y)$. The uniqueness of arcs in $M$ gives the first claim.

Fix $p\in\mathcal B(T)$. Distinct branches of $T$ at $p$ are mapped into distinct branches of $M$ at $f(p)$: otherwise, the unique target arc
joining points from two such images would avoid $f(p)$, whereas the first
part shows that it is the image of a source arc passing through $p$.
Thus, $f(p)$ is a branch point. Furthermore, the image of each source
branch has diameter at least $K^{-1}$ times the diameter of that branch.
At least three branches of $M$ at $f(p)$ therefore have diameter at least
$K^{-1}H_T(p)$, which proves the height estimate.
\end{proof}

We next define the gluing operation used in the construction of the
examples.  

Let $X$ and $Y_i$, $i\in I$, be geodesic metric spaces,
and choose points $x_i\in X$ and $y_i\in Y_i$.  Their
\emph{geodesic gluing}
\[
    (X,(x_i)_{i\in I})\bigvee_{i\in I}(Y_i,y_i)
\]
is the quotient of the disjoint union of $X$ and the $Y_i$ obtained by
identifying $y_i$ with $x_i$.  It is equipped with the path metric that
restricts to the original metric on each piece.  In particular, if
$z\in Y_i$, $w\in Y_j$, and $i\neq j$, then
\[
    d(z,w)
    =d_{Y_i}(z,y_i)+d_X(x_i,x_j)+d_{Y_j}(y_j,w),
\]
while, for $z\in Y_i$ and $w\in X$,
\[
    d(z,w)=d_{Y_i}(z,y_i)+d_X(x_i,w).
\]

\begin{lemma}[Lemma 11.3, \cite{CGIV2024}]\label{lem:glued_geodesic_metric}
Let $I$ be a countable infinite subset of $\N$. Let $X$ and $Y_i$ be geodesic metric trees for all $i \in I$ and let $x_i\in X, y_i \in Y$. Assume that either 
$\operatorname{card} I < \infty$, or $\operatorname{card} I = \infty$ and 
$\diam Y_i \to 0$ as $i \to \infty$. Then the geodesic gluing $
(X, d_X, x_i) \bigvee_{i\in I} (Y_i, d_{Y_i}, y_i)$
is a geodesic metric tree.
\end{lemma}

\subsection{Metric dimensions}

We use Hausdorff, Nagata, Lipschitz, and Assouad dimension, denoted by
\[
    \dim_H X,\qquad \dim_N X,\qquad \dim_L X,\qquad \adim  X,
\]
respectively.  We refer to \cite{book:hdim:rogers} for Hausdorff dimension, to
\cite{LangSchlichenmaier2005} for Nagata
dimension, to \cite{David:lip_dim21} for Lipschitz dimension, and to
\cite{fraser:ass} for Assouad dimension.

For clarity, we recall the two dimension facts needed below.  The Nagata
dimension is monotone under passage to subsets, and $\dim_N X\le\dim_L X$; see \cite[Corollary~3.5]{David:lip_dim21}. 

For a metric space $X$ and for $B\subset X$, let $N_X(B,r)$ represent the number of balls of radius $r$ required to cover $B$. The Assouad dimension of $X$ is the infimum of all $s\ge0$ for which there exists $C\ge1$ such that
\[
    N_X(B_X(x,R),r)
    \le C\left(\frac Rr\right)^s
\]
for every $x\in X$ and $0<r<R$ where $B_X(x,R)$ is a ball of radius $R$ centered at $x$.

\begin{theorem}[Theorem D, \cite{FreemanGartland:lip_qc_trees23}]\label{thm:freeman_gartland_lip}
    The Lipschitz dimension of any quasiconformal tree $T$ is equal to 1. Consequently, $\dim_N T =1$.
\end{theorem}

We finish with a covering lemma that combines an Assouad-dimension estimate
on a closed subset with uniform estimates at Whitney scales away from that
subset. The ideas of such decomposition have also been utilized to study conditions under which a metric space bi-Lipschitz embeds into a Euclidean space; see \cite{Jee11}.

\begin{lemma}\label{lem:whitney_scale_assouad}
Let $Y$ be a doubling metric space, let $X\subset Y$ be a nonempty
closed subset, and suppose that $\adim  X\le d$.  Assume that there is a
constant $C_0\ge1$ such that
\[
    N_Y(B_Y(y,R),r)
    \le C_0\left(\frac Rr\right)^d
\]
whenever $y\in Y\setminus X$, $0<R<\frac{\dist(y,X)}2$ and $0<r<R$.
Then $\adim  Y\le d$.
\end{lemma}

\begin{proof}
Fix $s>d$.  We prove an Assouad covering estimate with exponent $s$.
Let $y\in Y$ and $0<r<R$.  We first consider the case $y\in Y\setminus X$ and
\[
    R<\frac{\dist(y,X)}2,
\]
then the hypothesis away from $X$ gives
\[
    N_Y(B_Y(y,R),r)
    \le C_0\left(\frac Rr\right)^d
    \le C_0\left(\frac Rr\right)^s.
\]
Thus, in this case we are done.

\noindent
It remains to consider the case in which $y\in X$, or
$R\ge\dist(y,X)/2$.  We may then choose $y_0\in X$ such
that
\[
    d_Y(y,y_0)<3R.
\]
Consequently,
\[
    B_Y(y,R)\subseteq B_Y(y_0,4R).
\]

Set
\[
    E_0
    :=B_Y(y,R)\cap
      \{z\in Y:\dist(z,X)<8r\}.
\]
For $k\ge1$, set
\[
    r_k:=2^{k+2}r
\]
and
\[
    E_k
    :=B_Y(y,R)\cap
      \{z\in Y:r_k\le\dist(z,X)<2r_k\}.
\]
Then
\[
    B_Y(y,R)\subseteq E_0\cup\bigcup_{k\ge1}E_k.
\]
If $E_k\neq\varnothing$, then $r_k<4R$; hence only finitely many of
these layers are nonempty.

Since $\adim  X\le d<s$, there exists $C_s\ge1$ such that
\[
    N_Y(X\cap B_Y(y_0,12R),\rho)
    \le C_s\left(\frac{12R}{\rho}\right)^s
\]
for every $0<\rho<12R$.

For $z\in E_0$, choose $q_z\in X$ with $d_Y(z,q_z)<8r$.  The relevant
points $q_z$ lie in $X\cap B_Y(y_0,12R)$.  Covering this portion of
$X$ by radius-$r$ balls and enlarging them shows that $E_0$ is
covered by at most
\[
    C_s\left(\frac{12R}{r}\right)^s
\]
balls of radius $9r$.  Doubling each of these balls down to radius $r$
gives
\[
    N_Y(E_0,r)\le C_1\left(\frac Rr\right)^s,
\]
where $C_1$ depends only on $s$, $C_s$, and the doubling constant of
$Y$.

Now fix $k\ge1$ with $E_k\neq\varnothing$.  For every $z\in E_k$,
choose $q_z\in X$ such that $d_Y(z,q_z)<2r_k$.  Again,
$q_z\in X\cap B_Y(y_0,12R)$.  Covering this portion of $X$ by
radius-$r_k$ balls shows that $E_k$ is covered by at most
\[
    C_s\left(\frac{12R}{r_k}\right)^s
\]
balls of radius $3r_k$.  By doubling, these may be replaced by at most
\[
    C_2\left(\frac R{r_k}\right)^s
\]
balls of radius $r_k/8$, where $C_2$ is independent of $k$.

Retain only the balls that meet $E_k$, and choose a point
$z_i\in E_k$ in each retained ball.  The intersection of such a ball
with $E_k$ is contained in $B_Y(z_i,r_k/4)$.  Since
\[
    \frac{r_k}{4}
    <\frac{\dist(z_i,X)}2,
\]
the local covering hypothesis yields
\[
    N_Y(B_Y(z_i,r_k/4),r)
    \le C_0\left(\frac{r_k}{4r}\right)^d.
\]
It follows that
\[
    N_Y(E_k,r)
    \le C_3
       \left(\frac Rr\right)^s
       2^{-(s-d)k}.
\]

Summing the estimates for $E_0,E_1,E_2,\ldots$ gives
\[
    N_Y(B_Y(y,R),r)
    \le C_4\left(\frac Rr\right)^s.
\]
Since this holds for every $s>d$, we conclude that
$\adim  Y\le d$.
\end{proof}

\begin{remark}
    This lemma appears useful in other contexts. For instance, one can apply it to compute Assouad dimension of Grushin plane is $2$. We leave the details to an interested reader.
\end{remark}
\section{Lipschitz maps of intervals}\label{S:intervals_lip_maps}
The trees constructed in Section~\ref{S:T_a} are governed by an
$a$-dependent iterated subdivision of intervals. In this section, we
establish two estimates needed to compare these trees under
bi-Lipschitz embeddings. First, at sufficiently large subdivision
levels, all but a set of small relative measure is covered by
subintervals having the exponential length determined
by $a$. Second, if a Lipschitz map nearly maximally stretches the
endpoints of an interval, then, outside a set of small relative
measure, every subdivision interval containing a given point is
likewise nearly maximally stretched.

\subsection{The subdivision rule}
Fix $a\in [1/4, 1/3]$.

\begin{definition}[$a$-subdivision of $I$]\label{def:subdivision_rule}
    Let $I=[s,t]\subset\R$ with $s<t$. Define a family of subdivisions $\{\mathcal D_n^a(I)\}_{n\ge0}$ recursively as follows.
    Let $\mathcal D_0^a(I)=\{I\}$, and given $\mathcal D_n^a(I)$, define
    \[
    \mathcal D_{n+1}^a(I)
    =
    \bigcup_{J\in\mathcal D_n^a(I)}
    \Bigl\{
    [s_J,s_J+a|J|],\,
    [s_J+a|J|,t_J-a|J|],\,
    [t_J-a|J|,t_J]
    \Bigr\}.
    \]

\end{definition}
By induction, $\mathcal D_n^a(I)$ consists of $3^n$ closed
subintervals whose union is $I$ and whose interiors are pairwise
disjoint. Each $J\in\mathcal D_n^a(I)$ has length $|J|=|I|a^k(1-2a)^{n-k}$
for some $k\in\{0,\ldots,n\}$, where $k$ is the number of times the
subdivision path leading to $J$ chooses a side interval.

Moreover, every child has length at most one half the length of its
parent. Consequently, $\max\{|J|:J\in\mathcal D_n^a(I)\}\le 2^{-n}|I|$. In particular, the union of all subdivision endpoints is dense in $I$.

\subsection{Typical length of subintervals}
Define $\alpha_a := 2a\ln a + (1-2a)\ln(1-2a)$.
The function $a\mapsto\alpha_a$ is strictly decreasing on
$[1/4,1/3]$.

\begin{lemma}[typical length of subintervals]\label{lem:lawoflargenum}
For every $\tau,\eta>0$, there exists $N\in\N$ such that for all $n\ge N$,
    \[
    \mathcal L\,\left(
    \bigcup
    \Bigl\{
    J\in\mathcal D_n^a(I):
    \frac{|J|}{|I|}\ge e^{(\alpha+\tau)n}
    \ \text{or}\
    \frac{|J|}{|I|}\le e^{(\alpha-\tau)n}
    \Bigr\}
    \right)
    \le \eta |I|.
    \]
\end{lemma}
\begin{proof}
    For each $x\in I$, let $J_n(x)\in\mathcal D_n^a(I)$ be an interval at level $n$
    containing $x$. This interval is unique except for a finite collection of points. We equip $I$ with the normalized Lebesgue probability measure. Since $\mathcal D_n^a(I)$ is a partition of $I$, we have
    \[
        \mathbb P(J_n(x)=J)=\frac{|J|}{|I|}.
    \]
    At each subdivision step, a point $x$ lies in one of the two side intervals with total
    probability $2a$, and in the middle interval with probability $1-2a$.
    Define random variables $X_1,\dots,X_n:I\to\R$ by
    \[
        X_k(x)=
        \begin{cases}
        \ln a, & \text{if }x\text{ lies in a side interval at level }k,\\
        \ln(1-2a), & \text{if }x\text{ lies in the middle interval at level }k.
        \end{cases}
    \]
    Then the variables $X_k$ are independent and identically distributed, with
    \[
        \mathbb E[X_k]=2a\ln a+(1-2a)\ln(1-2a)=\alpha.
    \]
    By construction,
    \[
        |J_n(x)|=|I|\prod_{k=1}^n e^{X_k(x)},
    \]
    and therefore
    \[
    \ln\frac{|J_n(x)|}{|I|}=\sum_{k=1}^n X_k(x).
    \]
    Let
    \[
    \overline X_n(x)=\frac{1}{n}\sum_{k=1}^n X_k(x).
    \]
    By the weak law of large numbers, for every $\tau>0$,
    \[
        \mathbb P\,\left(
        \left\{x\in I:\bigl|\overline X_n(x)-\alpha\bigr|\ge\tau\right\}
        \right)
        \xrightarrow[n\to\infty]{}0.
    \]
    Equivalently,
    \[
        \frac{1}{|I|}
        \mathcal L\,\left(
        \left\{x\in I:\bigl|\overline X_n(x)-\alpha\bigr|\ge\tau\right\}
        \right)
        \xrightarrow[n\to\infty]{}0.
    \]
    Thus, for every $\eta>0$, there exists $N$ such that for all $n\ge N$,
    \[
        \mathcal L\,\left(
        \left\{x\in I:\bigl|\overline X_n(x)-\alpha\bigr|\ge\tau\right\}
        \right)
        \le \eta |I|.
    \]
    Finally, observe that
    \[
    \left\{x\in I:
    \frac{|J_n(x)|}{|I|}
    \notin
    \bigl[e^{(\alpha-\tau)n},e^{(\alpha+\tau)n}\bigr]
    \right\}
    =
    \left\{x\in I:
    \bigl|\overline X_n(x)-\alpha\bigr|\ge\tau
    \right\}.
    \]
    The set on the left is a disjoint union of those intervals
    $J\in\mathcal D_n^a(I)$ whose lengths lie outside the stated range.
    This yields the desired estimate.
\end{proof}

%\begin{remark}
%    If quantitative bounds are required, the weak law of large numbers may be replaced by Hoeffding's inequality.
%\end{remark}

\subsection{Lipschitz maps and almost maximally stretched intervals} 
Throughout this subsection, when we say that a map $f$ is $L$-Lipschitz, the constant $L>0$ is not necessarily the smallest possible Lipschitz constant unless explicitly specified.

\begin{lemma}\label{lem:shrinkingsmall}
    For every $\eta,\delta>0$ there exists $\varepsilon>0$ such that the following holds.
    
    Let $(M,d)$ be a metric space, let $I=[s_I,t_I]\subset\R$, where $s_I<t_I$, let $L>0$, and let $f:I\to M$ be $L$-Lipschitz. Assume that
    \[
        d\bigl(f(s_I),f(t_I)\bigr)
        \ge L(1-\varepsilon)|I|.
    \]
    Define
    \[
    B_{\delta,a}
    :=
    \Bigl\{x\in I:\ \exists\, n\ge 0,\ \exists\,J\in\mathcal D_n^a(I)\ \text{with }x\in J
    \ \text{and}\ d\bigl(f(s_J),f(t_J)\bigr)\le L(1-\delta)\,|J|\Bigr\}.
    \]
    Then $\mathcal L(B_{\delta,a})\le \eta |I|$.
\end{lemma}
\begin{proof}
    Let $\mathcal F$ denote the family of all intervals
    \[
    J\in\bigcup_{n\ge0}\mathcal D_n^a(I)
    \quad\text{such that}\quad
    d\bigl(f(s_J),f(t_J)\bigr)\le L(1-\delta)\,|J|.
    \]
    Then $B_{\delta,a}=\bigcup_{J\in\mathcal F} J$.
    
    Choose a subcollection $\{J_i\}_{i\in\mathcal I}\subset\mathcal F$ consisting of
    maximal elements with respect to inclusion.
    Since the collections $\mathcal D_n^a(I)$ form nested partitions of $I$,
    any two distinct maximal intervals $J_i$ and $J_j$ are disjoint up to end points. Hence
    \[
    \mathcal L(B_{\delta,a})=\sum_{i\in\mathcal I}|J_i|.
    \]

    Let $\mathcal I_0 \subset \mathcal I$ be a finite set. Order the intervals
    $\{J_i\}_{i\in\mathcal I_0}$ from left to right. Traversing $I$ from $s_I$
    to $t_I$ and applying the triangle inequality across these intervals and
    the complementary gaps gives
    \[
    d\bigl(f(s_I),f(t_I)\bigr)
    \le
    \sum_{i\in\mathcal I_0}
    d\bigl(f(s_{J_i}),f(t_{J_i})\bigr)
    +
    L\left(|I|-\sum_{i\in\mathcal I_0}|J_i|\right).
    \]

    Since $d\bigl(f(s_{J_i}),f(t_{J_i})\bigr) \le L(1-\delta)|J_i|$ and $d\bigl(f(s_I),f(t_I)\bigr)\ \ge\ L(1-\varepsilon)\,|I|$, we obtain
    \[
        L(1-\varepsilon) |I|\,\le\,d\bigl(f(s_I),f(t_I)\bigr)
        \,\le\,
        L|I|-L\delta\sum_{i\in\mathcal I_0}|J_i|
    \]
    and therefore
    \[
        \sum_{i\in\mathcal I_0}|J_i|
        \le
        \frac{\varepsilon}{\delta}|I|.
    \]
    Since this holds for every finite subcollection $\mathcal I_0\subset\mathcal I$,
    we get
    \[
        \sum_{i\in\mathcal I}|J_i|
        \le
        \frac{\varepsilon}{\delta}|I|.
    \]
    Thus $\mathcal L(B_{\delta,a}) \le \frac{\varepsilon}{\delta}|I$. Choosing $\varepsilon :=\eta\delta$ yields $\mathcal L(B_{\delta,a})\,\le\, \eta |I|$, the proof is complete.
    \end{proof}
In particular, if $x\in I\setminus B_{\delta,a}$, then every
$J\in\mathcal D_n^a(I)$, at every level $n\ge0$, that contains $x$
satisfies
\[
d\bigl(f(s_J),f(t_J)\bigr)>L(1-\delta)|J|.
\]

\begin{comment}
    
\begin{remark} % Remove this remark
    The conclusion of Lemma~\ref{lem:shrinkingsmall} holds for any nested family of subintervals of $I$; the dependence on the $a$-subdivision will become relevant only when combined with Lemma~\ref{lem:lawoflargenum}.
\end{remark}
\end{comment}

\section{A family of geodesic trees}\label{S:T_a}
In this section, we introduce a one-parameter family of geodesic trees $(T_a)_{a\in[1/4,1/3]}$ that will serve as the main objects in the proof of Theorem~\ref{thm:mainthm1}. The trees are obtained by an iterated subdivision of edges following the same ternary rule introduced
in Section~\ref{S:intervals_lip_maps}. This construction produces a nested
structure of branch points whose geometry is sensitive to the choice of 
parameter $a$.

\subsection{Construction of tree $\mathbf{T_a}$}

Fix $a\in [1/4,1/3]$. Let $c \in (0, 1)$ be fixed.
The restriction $a\in[1/4,1/3]$ ensures that the three subintervals produced at each step are uniformly comparable in length: if an interval $I$ is subdivided into pieces of lengths
\[
a|I|,\qquad (1-2a)|I|,\qquad a|I|,
\]
then
\[
a \le 1-2a \le 2a.
\]
Hence each child interval has length comparable to $a|I|$, with constants independent of the stage of the construction. This balanced subdivision will be used later when comparing heights and distances in the tree.

We construct $T_a$ iteratively. Begin with a line segment
\[
I_0=[s_{I_0},t_{I_0}]
\]
of length $1$. At each endpoint $s_{I_0}$ and $t_{I_0}$, attach two line segments of length $c|I_0|=c$. The resulting set $T_0^a$ consists of five line segments, two branch points, and four leaves, as shown in Figure~\ref{fig:T_0}.
\begin{figure}[h]
                \centering
        
                \tikzstyle arrowstyle=[scale=1]
                    \tikzstyle arrowtipinmiddle=[postaction={decorate,decoration={markings,mark=at position .56 with {\arrow[arrowstyle]{stealth'}}}}]
                    % Define parameters
                    \begin{tikzpicture}[scale=0.7]
                    
                    \node[fill=black, circle, inner sep=1.5pt, label=below:\footnotesize{$s_{I_0}$}] (x) at (0,5) {};
                    \node[fill=black, circle, inner sep=1.5pt, label=below:\footnotesize{$t_{I_0}$}] (y) at (10,5) {};
                    % \node[fill=black, circle, inner sep=1.5pt, label=below:\footnotesize{$p$}] (p) at (3.30,5) {};
                    % \node[fill=black, circle, inner sep=1.5pt, label=below:\footnotesize{$q$}] (q) at (6.8,5) {};
                   
                    \draw[fill=none, inner sep = 0.05cm] (-1.5,5.5)  node (w1) {};
                    \draw[fill=none, inner sep = 0.05cm] (-1.5,4.5)  node (w2) {};
            
                    \draw[fill=none, inner sep = 0.05cm] (11.5,5.5)  node (u1) {};
                    \draw[fill=none, inner sep = 0.05cm] (11.5,4.5)  node (u2) {};
                    
                    \draw[line width=0.02500cm, color=springgreen] (0,5) to (-1.5, 4.5);

                    \draw[line width=0.02500cm, color=springgreen, bend left=0] (x) to (w1);
                    \draw[ line width=0.02500cm, color=springgreen, bend left=0] (x) to (w1);
                    \draw[ line width=0.03500cm, color=royalblue, bend left=0] (y) to (x);

                    \draw[black!90, <->, dotted] (0,4.3) -- (10,4.3);
                    \node[fill=none, gray, circle, inner sep=1.5pt, label=below:\footnotesize{\textcolor{gray}{$1$}}] (l) at (5,4.4) {};
                    
                   % Right arcs
                    \draw[ line width=0.02500cm, color=springgreen, bend left=0] (y) to (u1);
                    \draw[ line width=0.02500cm, color=springgreen, bend left=0] (y) to (u2);

                    % \draw[line width=0.02500cm, color=springgreen] (3.3,5) to (3.3, 5.5);
                    % \draw[line width=0.02500cm, color=springgreen] (6.8,5) to (6.8, 5.5);

                    % \node[fill=none, circle, inner sep=1.5pt, label=below:\textcolor{orange}{\footnotesize{$a$}}] () at (1.7,6) {};

                    % \node[fill=none, circle, inner sep=1.5pt, label=below:\textcolor{orange}{\footnotesize{$(1-2a)$}}] () at (5,6) {};

                    % \node[fill=none, circle, inner sep=1.5pt, label=below:\textcolor{orange}{\footnotesize{$a$}}] () at (8.3,6) {};

                    \end{tikzpicture}

                \caption{$T_0^a$}
                \label{fig:T_0}
            \end{figure}
Next, divide $I_0$ into three subintervals
\[
I_{01}=[s_{I_{01}},t_{I_{01}}],\qquad
I_{02}=[s_{I_{02}},t_{I_{02}}],\qquad
I_{03}=[s_{I_{03}},t_{I_{03}}],
\]
of lengths $a|I_0|$, $(1-2a)|I_0|$, and $a|I_0|$, respectively, with
\[
t_{I_{01}}=s_{I_{02}},\qquad t_{I_{02}}=s_{I_{03}}.
\]
At each subdivision point $t_{I_{01}}$ and $t_{I_{02}}$, attach one
line segment of length $
c\,a\,|I_0|$. This produces the set $T_1^a$, shown in Figure~\ref{fig:T_1}.
  \begin{figure}
    \centering
    \tikzstyle arrowstyle=[scale=1]
        \tikzstyle arrowtipinmiddle=[postaction={decorate,decoration={markings,mark=at position .56 with {\arrow[arrowstyle]{stealth'}}}}]
        % Define parameters
        \begin{tikzpicture}[scale=0.7]
        
        \node[fill=black, circle, inner sep=1.5pt, label=below:\footnotesize{$x$}] (x) at (0,5) {};
        \node[fill=black, circle, inner sep=1.5pt, label=below:\footnotesize{$y$}] (y) at (10,5) {};
        \node[fill=black, circle, inner sep=1.5pt, label=below:\footnotesize{$p$}] (p) at (3.30,5) {};
        \node[fill=black, circle, inner sep=1.5pt, label=below:\footnotesize{$q$}] (q) at (6.8,5) {};
       
        \draw[fill=none, inner sep = 0.05cm] (-1.5,5.5)  node (w1) {};
        \draw[fill=none, inner sep = 0.05cm] (-1.5,4.5)  node (w2) {};

        \draw[fill=none, inner sep = 0.05cm] (11.5,5.5)  node (u1) {};
        \draw[fill=none, inner sep = 0.05cm] (11.5,4.5)  node (u2) {};
        
        \draw[line width=0.02500cm, color=springgreen] (0,5) to (-1.5, 4.5);

        \draw[line width=0.02500cm, color=springgreen, bend left=0] (x) to (w1);
        \draw[ line width=0.02500cm, color=springgreen, bend left=0] (x) to (w1);
        \draw[ line width=0.03500cm, color=royalblue, bend left=0] (y) to (x);
       
        \draw[ line width=0.02500cm, color=springgreen, bend left=0] (y) to (u1);
        \draw[ line width=0.02500cm, color=springgreen, bend left=0] (y) to (u2);

        \draw[line width=0.02500cm, color=springgreen] (3.3,5) to (3.3, 5.5);
        \draw[line width=0.02500cm, color=springgreen] (6.8,5) to (6.8, 5.5);

        \node[fill=none, circle, inner sep=1.5pt, label=below:\textcolor{orange}{\footnotesize{$a~L$}}] () at (1.7,6) {};

        \node[fill=none, circle, inner sep=1.5pt, label=below:\textcolor{orange}{\footnotesize{$(1-2a)~L$}}] () at (5,6) {};

        \node[fill=none, circle, inner sep=1.5pt, label=below:\textcolor{orange}{\footnotesize{$a~L$}}] () at (8.3,6) {};
        \end{tikzpicture}
    \caption{$T_1^a$}
    \label{fig:T_1}
\end{figure}
We continue inductively: at each stage, every distinguished interval $I$ is subdivided into three subintervals of lengths $a|I|$, $(1-2a)|I|$, and $a|I|$, and one new edge of length $c a |I|$ is attached at each of the two interior subdivision points. Note that only the distinguished intervals are subdivided in this process; the newly attached edges are never subdivided.

The sets $T_k^a$ form an increasing sequence, $T_0^a \subset T_1^a \subset T_2^a \subset \cdots$, and we define the tree
\begin{equation}
    T_a := \bigcup_{k=0}^{\infty} T_k^a,
\end{equation}
equipped with the induced path metric.

\begin{proposition}
    For every $a\in [1/4, 1/3]$, the space $T_a$, equipped with its path metric, is a complete geodesic metric tree. Moreover, $\Val(T_a)=3$, its branch points are dense in $I_0$, and, for all distinct $p,q\in \mathcal{B}(T_a)$,
    \begin{equation}\label{eq:unifrom-sep}
        d_{T_a}(p,q) \ge \frac{1}{c} \min \{H_{T_a}(p), H_{T_a}(q)\}.
    \end{equation}
    Consequently, $T_a$ is a quasiconformal tree and $T_a \in \mathcal{GT}(\bfn, \bfc)$ for every $n\ge 3$. 
\end{proposition}

\begin{proof}
    Let $\mathcal{A}$ be the collection of all edges attached during the construction of $T_q$. Each element of $\mathcal{A}$ is isometric to a compact interval and is attached to a point of $I_0$. For $Y\in\mathcal{A}$, let $m(Y)$ denote the stage at which $Y$ is attached. Since only finitely many edges are attached at each stage, we may enumerate $\mathcal{A}=\{Y_i\}_{i\in I}$, where $I\subset\N$, so that $m(Y_i)\le m(Y_j)$ whenever $i<j$. Since $\diam Y_i\to 0$ as $i\to \infty$, Lemma~\ref{lem:glued_geodesic_metric} implies that $T_a$ is a geodesic metric tree and hence, it is also complete. Furthermore, by construction, valency of each branch point is $3$.
    
    We next show that the branch points are dense in $I_0$. Let $x\in I_0$ and let $U$ be any open neighborhood of $x$. Since the distinguished intervals are subdivided at every stage and their lengths tend to $0$, there exists a stage $N$ and a distinguished interval $I\subset U$ containing $x$ such that $|I|$ is arbitrarily small. At the next stage, two new branch points are created at the subdivision points of $I$. Since these points lie in $I\subset U$, the neighborhood $U$ contains a branch point.

    Finally, the tree $T_a$ has uniform branch separation with constant $\bfc$. We choose the gluing parameter in the construction of $T_a$ to be $c := \bfc$. Let $p$ and $q$ be two distinct branch points of $T_a$. Without loss of generality, assume that $
    H_T(p)\le H_T(q).$ Then it suffices to show that $$ d(p,q)\ge \bfc\,H_T(p).$$

    First suppose that $p$ is an endpoint of $I_0$. Then $H_{T_a}(p)=c$. Every non-endpoint branch point has height at most $ca<c$. Hence the assumption $H_{T_a}(p)\le H_{T_a}(q)$ forces $q$ to be the other endpoint of $I_0$. This yields $d(p,q)=1\ge \bfc H_{T_a}(p)$.

    Now suppose that $p$ is created by subdividing a unique interval $J$. Then $H_T(p) = \bfc a |J|$. We claim that
    \[
        d_{T_a}(p,q)\ge a|J|.
    \]

Indeed, both children of $J$ adjacent to $p$ have length at least
$a|J|$, and the other subdivision point of $J$ is at distance
$(1-2a)|J|\ge a|J|$ from $p$. Therefore, if
$d_{T_a}(p,q)<a|J|$, then $q$ must be a branch point created later
inside a proper descendant $K$ of one of the children adjacent to
$p$. In that case,
\[
    H_{T_a}(q)=ca|K|<ca|J|=H_{T_a}(p),
\]
contrary to our choice of $p$. Thus
\[
    d_{T_a}(p,q)
    \ge a|J|
    =\frac{1}{c}H_{T_a}(p),
\]
which proves \eqref{eq:unifrom-sep}. In particular,
\[
    d_{T_a}(p,q)
    \ge c\min\{H_{T_a}(p),H_{T_a}(q)\},
\]
so the branch points are uniformly relatively separated with constant
$c$. Lemma~\ref{lem:doubling_from_branch_sep} therefore implies that
$T_a$ is doubling. Since every geodesic tree is $1$-bounded turning,
$T_a$ is a quasiconformal tree. In particular, $T_a$ belongs to the class $\mathcal{GT}(n,\bfc)$ where $n\ge 3$.
\end{proof}

\begin{proposition}
For every $a\in[1/4,1/3]$, the tree $T_a$ admits a bi-Lipschitz
embedding into $\R^2$. Moreover,
\[
    \dim_H(T_a)=\dim_N(T_a)=\adim (T_a)=1.
\]
\end{proposition}
\begin{proof}
Identify $I_0$ isometrically with $[0,1]\times\{0\}$. Map every edge
attached at a non-endpoint branch point isometrically onto the vertical
segment lying above its attachment point. At each endpoint of $I_0$,
map one of the two attached edges upward and the other downward. This
defines an injective map
\[
    F:T_a\longrightarrow \R^2.
\]
Because $F$ maps each edge isometrically to a line segment, the Euclidean
distance between any two points is at most the length of the image of the geodesic joining them. Hence
\[
    |F(x)-F(y)|\le d_{T_a}(x,y)
\]
for all $x,y\in T_a$.

We prove the reverse inequality. It is immediate if $x$ and $y$ lie on the same edge. If one point lies on $I_0$ and the other lies on an attached edge, then, for suitable $\alpha,\beta\ge0$,
\[
    d_{T_a}(x,y)=\alpha+\beta
    \le\sqrt2\,(\alpha^2+\beta^2)^{1/2}
    =\sqrt2\,|F(x)-F(y)|.
\]

Suppose now that $x$ and $y$ lie on distinct attached edges based at
$p,q\in I_0$, respectively. This means $p\neq q$. Now set
\[
    \alpha=d_{T_a}(x,p),\qquad
    \gamma=d_{T_a}(y,q),\qquad
    \beta=d_{T_a}(p,q)
\]
as shown in Figure~\ref{fig:biLip}.
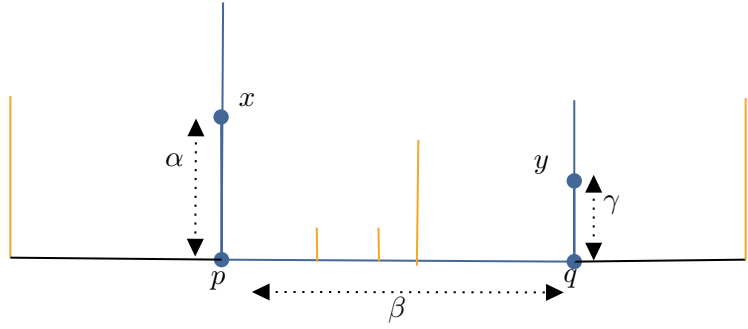
\begin{figure}[h]
    \centering
    \tikzset{every picture/.style={line width=0.75pt}} %set default line width to 0.75pt       
    \begin{tikzpicture}[x=0.75pt,y=0.75pt,yscale=-1,xscale=1]
    %uncomment if require: \path (0,832); %set diagram left start at 0, and has height of 832
    
    %Straight Lines [id:da974074339740122] 
    \draw [color={rgb, 255:red, 67; green, 104; blue, 152 }  ,draw opacity=1 ]   (221.83,705.3) -- (398.83,706.3) ;
    \draw [shift={(398.83,706.3)}, rotate = 0.32] [color={rgb, 255:red, 67; green, 104; blue, 152 }  ,draw opacity=1 ][fill={rgb, 255:red, 67; green, 104; blue, 152 }  ,fill opacity=1 ][line width=0.75]      (0, 0) circle [x radius= 3.35, y radius= 3.35]   ;
    \draw [shift={(221.83,705.3)}, rotate = 0.32] [color={rgb, 255:red, 67; green, 104; blue, 152 }  ,draw opacity=1 ][fill={rgb, 255:red, 67; green, 104; blue, 152 }  ,fill opacity=1 ][line width=0.75]      (0, 0) circle [x radius= 3.35, y radius= 3.35]   ;
    %Straight Lines [id:da4837302029328334] 
    \draw    (398.83,706.3) -- (484.83,705.3) ;
    %Straight Lines [id:da6284501021953278] 
    \draw [color={rgb, 255:red, 67; green, 104; blue, 152 }  ,draw opacity=1 ][line width=0.75]    (222.58,576.25) -- (221.83,705.3) -- (221.58,633.75) ;
    \draw [shift={(221.58,633.75)}, rotate = 269.8] [color={rgb, 255:red, 67; green, 104; blue, 152 }  ,draw opacity=1 ][fill={rgb, 255:red, 67; green, 104; blue, 152 }  ,fill opacity=1 ][line width=0.75]      (0, 0) circle [x radius= 3.35, y radius= 3.35]   ;
    %Straight Lines [id:da3472013889325124] 
    \draw [color={rgb, 255:red, 245; green, 166; blue, 35 }  ,draw opacity=1 ]   (320.58,645.25) -- (319.83,708.3) ;
    %Straight Lines [id:da052602299411180886] 
    \draw [color={rgb, 255:red, 67; green, 104; blue, 152 }  ,draw opacity=1 ][line width=0.75]    (398.83,625.3) -- (398.83,706.3) -- (398.83,665.8) ;
    \draw [shift={(398.83,665.8)}, rotate = 270] [color={rgb, 255:red, 67; green, 104; blue, 152 }  ,draw opacity=1 ][fill={rgb, 255:red, 67; green, 104; blue, 152 }  ,fill opacity=1 ][line width=0.75]      (0, 0) circle [x radius= 3.35, y radius= 3.35]   ;
    %Straight Lines [id:da606968543071433] 
    \draw [color={rgb, 255:red, 245; green, 166; blue, 35 }  ,draw opacity=1 ]   (269.58,689.25) -- (269.83,706.3) ;
    %Straight Lines [id:da8481083612661293] 
    \draw [color={rgb, 255:red, 245; green, 166; blue, 35 }  ,draw opacity=1 ]   (115.83,623.3) -- (115.83,704.3) ;
    %Straight Lines [id:da26261046719143366] 
    \draw  [dash pattern={on 0.84pt off 2.51pt}]  (208.98,637.5) -- (208.6,700.75) ;
    \draw [shift={(208.58,703.75)}, rotate = 270.34] [fill={rgb, 255:red, 0; green, 0; blue, 0 }  ][line width=0.08]  [draw opacity=0] (8.93,-4.29) -- (0,0) -- (8.93,4.29) -- cycle    ;
    \draw [shift={(209,634.5)}, rotate = 90.34] [fill={rgb, 255:red, 0; green, 0; blue, 0 }  ][line width=0.08]  [draw opacity=0] (8.93,-4.29) -- (0,0) -- (8.93,4.29) -- cycle    ;
    %Straight Lines [id:da5221741681815085] 
    \draw  [dash pattern={on 0.84pt off 2.51pt}]  (408.58,665.75) -- (408.58,702.75) ;
    \draw [shift={(408.58,705.75)}, rotate = 270] [fill={rgb, 255:red, 0; green, 0; blue, 0 }  ][line width=0.08]  [draw opacity=0] (8.93,-4.29) -- (0,0) -- (8.93,4.29) -- cycle    ;
    \draw [shift={(408.58,662.75)}, rotate = 90] [fill={rgb, 255:red, 0; green, 0; blue, 0 }  ][line width=0.08]  [draw opacity=0] (8.93,-4.29) -- (0,0) -- (8.93,4.29) -- cycle    ;
    %Straight Lines [id:da33682412887792756] 
    \draw [color={rgb, 255:red, 245; green, 166; blue, 35 }  ,draw opacity=1 ]   (300.58,689.25) -- (300.83,706.3) ;
    %Straight Lines [id:da6272231165456011] 
    \draw  [dash pattern={on 0.84pt off 2.51pt}]  (240,721.5) -- (390.58,721.75) ;
    \draw [shift={(393.58,721.75)}, rotate = 180.09] [fill={rgb, 255:red, 0; green, 0; blue, 0 }  ][line width=0.08]  [draw opacity=0] (8.93,-4.29) -- (0,0) -- (8.93,4.29) -- cycle    ;
    \draw [shift={(237,721.5)}, rotate = 0.09] [fill={rgb, 255:red, 0; green, 0; blue, 0 }  ][line width=0.08]  [draw opacity=0] (8.93,-4.29) -- (0,0) -- (8.93,4.29) -- cycle    ;
    %Straight Lines [id:da6843587072590259] 
    \draw    (115.83,704.3) -- (221.83,705.3) ;
    %Straight Lines [id:da5791650269156297] 
    \draw [color={rgb, 255:red, 245; green, 166; blue, 35 }  ,draw opacity=1 ]   (484.83,624.3) -- (484.83,705.3) ;
    
    % Text Node
    \draw (215,709.5) node [anchor=north west][inner sep=0.75pt]    {$p$};
    % Text Node
    \draw (392,708.5) node [anchor=north west][inner sep=0.75pt]    {$q$};
    % Text Node
    \draw (229,620.5) node [anchor=north west][inner sep=0.75pt]    {$x$};
    % Text Node
    \draw (377,651.5) node [anchor=north west][inner sep=0.75pt]    {$y$};
    % Text Node
    \draw (192,650.5) node [anchor=north west][inner sep=0.75pt]    {$\alpha $};
    % Text Node
    \draw (304,722.5) node [anchor=north west][inner sep=0.75pt]    {$\beta $};
    % Text Node
    \draw (412,671.5) node [anchor=north west][inner sep=0.75pt]    {$\gamma $};

    \end{tikzpicture}
    \caption{$ \alpha=d_T(p,x), \gamma=d_T(q,y), \beta=d_T(p,q)$}
    \label{fig:biLip}
\end{figure}
Let $d_E$ denote Euclidean distance in $\mathbb{R}^2$, and let $d_T$ denote the
geodesic distance in $T_a$. Then
\[
d_T(x,y)=\alpha+\beta+\gamma, \qquad
|F(x)-F(y)|^2=(\alpha-\gamma)^2+\beta^2.
\]

We may assume without loss of generality that the branch containing $y$ has smaller
height, that is $H_{T_a}(q) \le H_{T_a}(p)$. By \eqref{eq:unifrom-sep},
\[
\gamma \le c\beta.
\]

\vspace{0.1cm}
\noindent
\textit{Case 1: $\alpha\le \gamma$.}
Then $\alpha\le\gamma\le c\beta$, and hence
\[
    d_T(x,y)=\alpha+\beta+\gamma\le(2c+1)\beta.
\]
Since
\[
    |F(x)-F(y)|^2=(\alpha-\gamma)^2+\beta^2\ge\beta^2,
\]
we obtain $|F(x)-F(y)|\ge\beta$, and therefore
\[
    d_T(x,y)\le(2c+1)|F(x)-F(y)|.
\]

\vspace{0.1cm}
\noindent
\textit{Case 2: $\alpha>\gamma$.}
Then $d_T(x,y)=\alpha+\beta+\gamma=(\alpha-\gamma)+\beta+2\gamma$. Moreover, we have
\[
|F(x)-F(y)|\ge\alpha-\gamma
\qquad\text{and}\qquad
|F(x)-F(y)|\ge\beta.
\]
Using $\gamma\le c\beta$, we obtain
\begin{align*}
    d_T(x,y) & \,\le\,(\alpha-\gamma)+(1+2c)\beta\\
            & \,\le \,|F(x)-F(y)|+(1+2c)|F(x)-F(y)|\\
 & \,=\,(2+2c)|F(x)-F(y)|.
\end{align*}
Thus in all cases
\[
|F(x)-F(y)|\le d_T(x,y)\le (2+2c)|F(x)-F(y)|,
\]
so the embedding is bi-Lipschitz.

Since $T_a$ is a countable union of line segments, its Hausdorff dimension
satisfies $\dim_H(T_a)\le 1$, while the initial segment $I_0\subset T_a$
implies $\dim_H(T_a)\ge 1$. Hence $\dim_H(T_a)=1$. Furthermore, since $T_a$ is a geodesic quasiconformal tree, by
Theorem~\ref{thm:freeman_gartland_lip} we have $\dim_L(T_a)=1$.
In particular, this implies that $1=\dim_N I_0 \le \dim_N(T_a) \le \dim_L(T_a)=1$. 

Finally, for Assouad dimension, we apply Lemma~\ref{lem:whitney_scale_assouad} with $X=I_0$, $Y=T_a$ and $d=1$. The set $I_0$ is closed, $\adim  I_0=1$, and $T_a$ is doubling. Now, let $x\in T_a\setminus I_0$ and suppose $$0<R<\frac{\dist(x, I_0)}{2}.$$
Then $x$ belongs to a unique attached edge, and $B_{T_a}(x,R)$ is contained in that edge since $T_a$ is a geodesic tree. Since this ball is an interval of length at most $2R$, for every $0<r<R$,
\[
    N_{T_a}(B_{T_a}(x,R),r)
    \le3\frac Rr.
\]
Thus all the hypotheses of Lemma~\ref{lem:whitney_scale_assouad} are
satisfied, and $\adim  I_0 = 1$. 
\end{proof}

\section{Proof of Theorem~\ref{thm:mainthm1}}\label{S:mainthmproof}
In this section we prove Theorem~\ref{thm:mainthm1}. The argument proceeds by comparing two bi-Lipschitz embeddings $f_{a}$ and $f_{b}$ corresponding to nearby parameters $a,b\in[1/4,1/3]$. We begin by identifying admissible horizontal intervals in $T_{a}$ and $T_{b}$ which have typical length and are well-stretched as discussed in Section~\ref{S:intervals_lip_maps}. We then show that their endpoints must share the same images in the target tree $M$ and Lemma~\ref{lem:maxstretch} forces agreement of the images of first-generation subdivision points. Iterating along well-stretched descendants, and invoking the typical-length lemma from Section~\ref{S:intervals_lip_maps}, yields the desired contradiction. 

We introduce the necessary notation and terminology in Subsection~\ref{ss:notation5}, establish auxiliary lemmas in Subsection~5.2, a family of parameters which produces contradiction in Subsection~5.3, and complete the proof in Subsection~5.4.

\subsection{Setup and terminology}\label{ss:notation5}

Throughout this section, fix $\bfn \in \N$ and $\bfc \in (0,1)$. We obtain a family $$ \{T_a \mid a\in [1/4,1/3] \text{ and } T_a \in \mathcal{GT}(\bfn, \bfc)\}$$ as constructed in Section~\ref{S:T_a}. Our goal is to prove that $\mathcal{GT}(\bfn, \bfc)$ has no bi-Lispchitz universal element.

Now we define some terminologies associated to $T_a$ for every $ a\in [1/4,1/3]$.
\begin{definition}[Horizontal segment]
Let $H_a\subset T_a$ denote the distinguished interval obtained from the initial segment $[0,1]$ by recursive subdivision in the construction of $T_a$. We refer to $H_a$ as the \emph{horizontal segment} of $T_a$.
\end{definition}

\begin{notation}
    For a horizontal interval $J\subset H_a$ and $n\ge 0$, we denote by $\mathcal I_n^a(J)$ the set of endpoints of intervals in $\mathcal D_n^a(J)$.
\end{notation}

\begin{definition}[Horizontal Lipschitz constants]
Let $f_a:T_a\to M$ be a globally $\left(\frac{1}{L'},L'\right)$-bi-Lipschitz embedding. We say that $L$ is a \emph{horizontal Lipschitz constant} for $f_a$ if
\[
d_M\bigl(f_a(x_a),f_a(y_a)\bigr)\le L\, d_{T_a}(x_a,y_a)\qquad\forall\, x_a,y_a\in H_a.
\]
Equivalently,
\[
\Lip(f_a|_{H_a})\le L.
\]
\end{definition}

\begin{definition}[Adjacent branch points]
Two branch points $b,b'\in H_a$ are said to be \emph{adjacent} if there exists an iteration $k\ge 0$ such that $b$ and $b'$ are the endpoints of a single edge of $T_k^a$.
\end{definition}

\begin{definition}[$(L,\delta)$--well-stretched intervals]\label{def:wellstretched}
Let $I=[s_I,t_I]\subset\R$, let $(X,d)$ be a metric space, let $f:I\to X$ be $L$-Lipschitz, and let $0<\delta<1$. For $n\ge 0$ and $J\in \mathcal D_n^a(I)$, we say that $J$ is \emph{$(L,\delta)$--well-stretched} if
\[
d\bigl(f(s_J),f(t_J)\bigr)>L(1-\delta)\,|J|.
\]
\end{definition}
Recall $\alpha_a = 2a\ln a + (1-2a)\ln(1-2a)$.
\begin{definition}[Admissible intervals]\label{def:admissible}
Fix $n\ge 0, L\geq 1$, a biLipschitz embeddging $f_a:T_a\to M$ and a small $\tau>0$. Let $J_a\subset H_a$ be a horizontal interval. An interval $I\in\mathcal D_n^a(J_a)$ is called \emph{admissible} if: 
\begin{enumerate}
    \item[\textup{(i)}] $I$ is $(L,\delta')$--well-stretched, i.e.
    \[
    d\bigl(f_a(s_I),f_a(t_I)\bigr)>L(1-\delta')\,|I|;
    \]
    \item[\textup{(ii)}] $I$ is typical, i.e.
    \[
    e^{(\alpha_a-\tau)n}\le \frac{|I|}{|J_a|}\le e^{(\alpha_a+\tau)n}.
    \]
\end{enumerate}
An interval $I\in \mathcal D_n^a(J_a)$ is called \emph{non-typical} if
\[
|I|\notin [e^{(\alpha_a-\tau)n}|J_a|,\; e^{(\alpha_a+\tau)n}|J_a|].
\]
\end{definition}

In the proof of Theorem~\ref{thm:mainthm1}, we will invoke overlapping of admissible intervals of two distinct trees $T_a$ and $T_b$ for $a,b$ very close to reach contradiction. 

\begin{definition}[Combinatorial address]
Let $J\subset H_a$ be a horizontal interval. For each $n\ge 0$, every interval
in $\mathcal D_n^a(J)$ is uniquely determined by a word
\[
\omega=(\omega_1,\dots,\omega_n)\in\{1,2,3\}^n,
\]
where $\omega_k=1,2,3$ records whether at the $k$-th subdivision step one chooses
the left, middle, or right child, respectively.

We denote by
\[
I_\omega^a(J)\in \mathcal D_n^a(J)
\]
the unique level-$n$ interval corresponding to the address $\omega$.
When the ambient interval $J$ is clear from context, we simply write $I_\omega^a$.
\end{definition}
Intervals $I_\omega^a\subset J_a$ and $I_\omega^b\subset J_b$ with the same word
$\omega\in\{1,2,3\}^n$ will be called \emph{corresponding intervals}.

\subsection{Auxiliary lemmas}
In this subsection we prove two auxiliary lemmas needed for the proof of Theorem~\ref{thm:mainthm1}.

The first lemma establishes the mapping of endpoints from two trees $T_a$ and $T_b$ onto the common points in a traget tree $M$ under additional hypothesis. The second lemma establishes existence of well-stretched intervals in a trees $T_a$ under a bi-Lispchitz embedding.

\begin{lemma}\label{lem:maxstretch}
Let $M \in \mathcal{GT}(\bfn_M,\bfc_M)$. For $L>0$ and $L' \ge 1$, set
\[
    r(L,L'):=\min\left\{\frac12, \frac{\bfc \bfc_M}{24 L L'}\right\}.
\]
Let $0<\delta<r(L,L')$. Let $a,b\in[1/4,1/3]$ satisfy $|a-b|<\delta,$ and let
\[
    f_a:T_a\to M,\qquad f_b:T_b\to M
\]
be $\left(\frac{1}{L'},L'\right)$-bi-Lipschitz embeddings such that
$\Lip(f_a|_{H_a})\le L$, and $\Lip(f_b|_{H_b})\le L$.

Let $x_a<y_a$ and $x_b<y_b$ be ordered adjacent branch points such that
\[
    f_a(x_a)=f_b(x_b)=m,
    \qquad
    f_a(y_a)=f_b(y_b)=m',
\]
and assume that
\begin{align*}
    d_M(m,m')
    &>
    L(1-\delta)d_{T_a}(x_a,y_a), \text{ and }\\
        d_M(m,m')
    &>
    L(1-\delta)d_{T_b}(x_b,y_b).
\end{align*}
Write
\[
    \mathcal I_1^a([x_a,y_a])\setminus\{x_a,y_a\}
    =
    \{m_a^1,m_a^2\},
\]
\[
    \mathcal I_1^b([x_b,y_b])\setminus\{x_b,y_b\}
    =
    \{m_b^1,m_b^2\},
\]
where each pair is ordered from left to right. Then
\[
    f_a(m_a^i)=f_b(m_b^i),
    \qquad i=1,2.
\]
\end{lemma}

\begin{proof}
Let $a,b\in[1/4,1/3]$ with $|a-b|\le \delta$ and let $f_a:T_a\to M$ and $f_b:T_b\to M$ satisfy the hypotheses of the lemma. Denote 
\[
    \ell_a:=d_{T_a}(x_a,y_a),
    \,\,\,
    \ell_b:=d_{T_b}(x_b,y_b),
    \,\,\,
    \bfx := f_a(x_a) = f_b(x_b),
    \,\,\,
    \bfy:= f_a(y_a) = f_b(y_b).
\]   
By Lemma~\ref{lem:tree_embedding_arcs_heights},
\[
f_a([x_a,y_a])
=
[m,m']_M
=
f_b([x_b,y_b]).
\] 
Using the hypotheses and the horizontal $L$-Lipschitz bounds, we have
\[
    L(1-\delta)\ell_a<\, d_M\bigl(\bfx,\bfy\bigr) \, \le  L\ell_b
\]
and
\[
    L(1-\delta)\ell_b < \, d_M\bigl(\bfx,\bfy\bigr) \, \le  L\ell_a.
\]
Therefore,
\begin{equation}\label{eq:length_ratio}
    1-\delta
    \le 
    \frac{\ell_b}{\ell_a}
    \le 
    \frac{1}{1-\delta}.
\end{equation}

Let $m_a^1\in \mathcal I_1^a([x_a,y_a])\setminus\{x_a,y_a\}$ be the unique point satisfying $d(x_a,m_a^1)=a\,\ell_a$, and define $m_b^1\in \mathcal I_1^b([x_b,y_b])\setminus\{x_b,y_b\}$ analogously by $d(x_b,m_b^1)=b\,\ell_b$.

Since $f_a$ is $L$-Lipschitz,
\begin{equation}\label{eq:upper_a}
d\bigl(\bfx,f_a(m_a^1)\bigr)\ \le\ L a\,\ell_a.
\end{equation}

On the other hand,  using the maximal stretch of $[x_a, y_a]$ and the $L$-Lipschitz bound for the segment $[m_a^1, y_a]$, we obtain
\begin{align}
d\bigl(\bfx,f_a(m_a^1)\bigr)\ & \ge\  d\bigl(\bfx,\bfy\bigr) - d\bigl(f_a(m_a^1), \bfy\bigr) \nonumber\\
    & \ge \ L(1-\delta) \ell_a - L(1-a)\ell_a = L(a -\delta)\ell_a\label{eq:lower_a}
\end{align}

Combining \eqref{eq:upper_a} and \eqref{eq:lower_a}, we get
\begin{equation*}
    L(a-\delta)\ell_a\  \le \ d\bigl(\bfx,f_a(m_a^1)\bigr)\ \le \ La \ell_a.
\end{equation*}
Using \eqref{eq:length_ratio} and since $\delta \le 1/2$, analogously, for $T_b$, we obtain
\begin{equation*}
    L(b-\delta)\ell_b\  \le \ d\bigl(\bfx,f_b(m_b^1)\bigr)\ \le \ Lb \ell_b \ \le\ \frac{Lb}{1-\delta}\,\ell_a.
\end{equation*}

Without loss of generality, assume that $\ell_a \le \ell_b$, we may interchange $a$ and $b$ if necessary. Let the upper bound for target distances be $La \ell_a$ and $Lb \ell_b$, then
\begin{align*}
    |La \ell_a-Lb \ell_b|
    &=L\bigl|a\,\ell_a -b\,\ell_b\bigr|\\
    &\le L|a-b|\,\ell_a+Lb\,\bigl|\ell_a-\ell_b\bigr|\\
    &\le L\delta\,\ell_a+L\cdot \frac{1}{3}\cdot \frac{\delta}{1-\delta}\,\ell_a\\
    &\le L\delta\,\ell_a+L\cdot \frac{1}{3}\cdot 2\delta\,\ell_a\\
    &= \frac53 L\delta\,\ell_a
    \le 2L\delta\,\ell_a.
\end{align*}
Assume for contradiction that $f_a(m_a^1)\neq f_b(m_b^1)$. Since $M$ is a geodesic tree, $$d\bigl(f_a(m_a^1),f_b(m_b^1)\bigr) = \bigl|d\bigl(\bfx, f_a(m_a^1)\bigr) \ - \ d\bigl(\bfx, f_a(m_b^1)\bigr)\bigr|.$$ 
\begin{align}
    d\bigl(f_a(m_a^1),f_b(m_b^1)\bigr)
    &\le \bigl|d(\bfx,f_a(m_a^1))-La\ell_a\bigr|
         +|La\ell_a-Lb\ell_b|
         +\bigl|Lb\ell_b-d(\bfx,f_b(m_b^1))\bigr|\notag\\
    &\le 2L\delta\,\ell_a+2L\delta\,\ell_a+2L\delta\,\ell_a\notag\\
    &\le 6L\delta\,\ell_a\label{eq:image_upper_bound}.
    \end{align}

    % Finally, since $d(\bfx,\bfy)\ge L(1-\delta)d(x_a,y_a)$, we have
    % \begin{equation}\label{eq:close_midpoints}
    % d\bigl(f_a(m_a^1),f_b(m_b^1)\bigr)\ \le\ 12 \delta\, d(\bfx, \bfy).
    % \end{equation}

    Moreover, $T_a, T_b  \in \mathcal{GT}(\bfn,\bfc)$. Since both points are branch points of $M$ and $M  \in \mathcal{GT}(\bfn_M,\bfc_M)$, the uniform relative separation condition gives
    \begin{equation}\label{eq:image_lower_bound}
        d\bigl(f_a(m_a^1),f_b(m_b^1)\bigr) \ge \bfc_M \;\min\bigl\{H_M(f_a(m_a^1)), H_M(f_b(m_b^1))\bigr\}.
    \end{equation}
    By the construction of $T_a$, we obtain
    \[
    H_M(f_a(m_a^1))\ge \frac{1}{L'}\,H_{T_a}(m_a^1) =\frac{1}{L'} \bfc a \,d(x_a,y_a)\, \ge \frac{\bfc }{4L'}\ell_a.
    \]
    Similarly,
    \[
        H_M(f_b(m_b^1))
        \ge
        \frac{1}{L'}H_{T_b}(m_b^1)
        =
        \frac{\bfc b}{L'}d_{T_b}(x_b,y_b)
        \ge
        \frac{\bfc}{4L'}\ell_b
        \ge
        \frac{\bfc}{4L'}\ell_a,
    \]
    where the last inequality follows from our assumption that $\ell_a \le \ell_b$.
    Therefore,
    \[
    \min\left\{
        H_M(f_a(m_a^1)),
        H_M(f_b(m_b^1))
    \right\}
    \ge
    \frac{\bfc}{4L'}d_{T_a}(x_a,y_a).
    \]
    
    Combining \eqref{eq:image_lower_bound} and \eqref{eq:image_upper_bound}, we have 
    \begin{equation*}
        \frac{\;\bfc \bfc_M}{4L'} \le 6L\delta.
    \end{equation*}
    Thus, if $\delta<\frac{\bfc \bfc_M}{24\,LL'}$,
    we obtain a contradiction. Therefore
    $
    f_a(m_a^1)=f_b(m_b^1)
    $.
    Applying the same argument from the common right endpoint $m'$ to the
second subdivision points $\bfy$ in place of $\bfx$ gives $f_a(m_a^2)=f_b(m_b^2)$. 
\end{proof}

\begin{lemma}\label{lem:adjacent_horizontal_pair}
Let $f_a:T_a\to M$ be an $L$-Lipschitz embedding, and set
\[
\lambda_a:=\Lip_{\rm{hor}}(f_a)=\Lip(f_a|_{H_a}).
\]
Then for every $\varepsilon>0$ there exist adjacent branch points $x_a,y_a\in H_a$ such that
\[
\frac{d_M(f_a(x_a),f_a(y_a))}{d_{T_a}(x_a,y_a)}>\lambda_a-\varepsilon.
\]
\end{lemma}

\begin{proof}
By the definition of $\lambda_a$, there exist points $x<y$ in $H_a$ such that
\[
\frac{d_M(f_a(x),f_a(y))}{d_{T_a}(x,y)}>\lambda_a-\frac{\varepsilon}{2}.
\]
Since the subdivision endpoints are dense in $H_a$, we may choose subdivision endpoints
\[
s,t\in\bigcup_{n\ge 0}\mathcal I_n^a(H_a),\qquad s<t,
\]
such that
\[
\frac{d_M(f_a(s),f_a(t))}{d_{T_a}(s,t)}>\lambda_a-\varepsilon.
\]
Choose $n$ large enough so that $s,t\in\mathcal I_n^a(H_a)$. Then the segment $[s,t]\subset H_a$ is a finite union of level-$n$ subdivision intervals:
\[
[s,t]=J_1\cup\cdots\cup J_N,\qquad J_i=[u_{i-1},u_i],
\]
where each $J_i\in\mathcal D_n^a(H_a)$, and hence each $J_i$ joins adjacent branch points in $H_a$.

If every $J_i$ satisfied
\[
\frac{d_M(f_a(u_{i-1}),f_a(u_i))}{d_{T_a}(u_{i-1},u_i)}\le\lambda_a-\varepsilon,
\]
then by the triangle inequality,
\[
d_M(f_a(s),f_a(t))
\le
\sum_{i=1}^N d_M(f_a(u_{i-1}),f_a(u_i))
\le
(\lambda_a-\varepsilon)\sum_{i=1}^N d_{T_a}(u_{i-1},u_i)
=
(\lambda_a-\varepsilon)\,d_{T_a}(s,t),
\]
contradicting the choice of $s,t$. Therefore, for at least one $i$,
\[
\frac{d_M(f_a(u_{i-1}),f_a(u_i))}{d_{T_a}(u_{i-1},u_i)}>\lambda_a-\varepsilon.
\]
Taking $x_a=u_{i-1}$, $y_a=u_i$ proves the lemma.
\end{proof}

\subsection{Family of parameters and covers}\label{subsec:family_parameters}
Let $n_M\ge3$, let $c_M\in(0,1)$, and let
\[
    M\in\mathcal{GT}(n_M,c_M).
\]
Fix $\eta\in(0,\tfrac{1}{100})$. For
$L\in\Q_{>0}$ and $L'\in\Q_{\ge1}$, set
\[
    \delta_0(L,L')
    :=
    \eta\frac{\bfc c_M}{100LL'}.
\]
For every $L'\in\Q_{\ge1}$, set
\[
    \delta_{L'}
    :=
    \delta_0(L',L')
    =
    \eta\frac{\bfc c_M}{100(L')^2}.
\]
In particular, if $L<L'$, then
\[
    0<\delta_{L'}<\delta_0(L,L').
\]

\noindent
\underline{First family}: For $L \in \Q_{>0}$ and $L'\in\Q_{\ge 1}$, define
\begin{align}\label{eq:ALL'}
    A_{L,L'}
    :=
    \Bigl\{
    a\in[1/4,1/3]:
    &\ \exists\, f_a:T_a\to M
    \text{ a }(1/L',L')\text{-bi-Lipschitz embedding such that} \nonumber\\
    &\quad
    L(1-\delta_{L'})
    <
    \Lip(f_a|_{H_a})
    \le L
    \Bigr\}.
\end{align}
We claim that every parameter $a$ for which $T_a$ bi-Lipschitz embeds into $M$
belongs to some $A_{L,L'}$ with $L<L'$. 

Indeed, suppose $f_a:T_a\to M$ is bi-Lipschitz and let $\lambda_a := \Lip(f_a|_{H_a})$. Choose $L'\in\Q_{\ge 1}$ so large that $f_a$ is $(1/L',L')$-bi-Lipschitz, and $\lambda_a< L'$. We now choose a rational $L\ge \lambda_a$ sufficiently close to $\lambda_a$ so that
\[
    L(1-\delta_{L'})<\lambda_a \le L<L' .
\]
Such an $L$ exists by density of $\Q$ and therefore, $a\in A_{L,L'}$.

\begin{comment}    
\begin{align}\label{eq:ALL'}
A_{L,L'}
:=
\Bigl\{
a\in[1/4,1/3]:
\exists\,f_a:T_a\to M 
\text{ a }(1/L',L')\text{-bi-Lipschitz}\nonumber\\
\text{embedding such that }
\Lip_{\mathrm{hor}}(f_a)\le L
\Bigr\}.
\end{align}
Thus $A_{L,L'}$ consists of those parameters $a$ for which there exists an embedding
that is globally bi-Lipschitz with constants $\left(\frac{1}{L'},L'\right)$, and
whose restriction to the horizontal segment is $L$-Lipschitz.

Let $\delta_0:=\delta_0(L,L')>0$ be the constant given by Lemma~\ref{lem:maxstretch}.

 For each $L_*\in\Q\cap[1/L',L]$ and branch points $m,m'\in\mathcal{B}(M)$, define

\begin{align}\label{eq:ALL'mm'}
A^{L,L'}_{L_*,m,m'}
:=
\Bigl\{
a\in[1/4,1/3]:
\exists\, f_a:T_a\to M \text{ such that }
f_a \text{ is } (1/L',L')\text{-bi-Lipschitz}, \nonumber\\
\Lip_{\mathrm{hor}}(f_a)\le L,
\exists\,\text{adjacent branch points }x_a,y_a\in H_a \text{ with } \nonumber\\
f_a(x_a)=m, f_a(y_a)=m',
d_M(f_a(x_a),f_a(y_a))\ge L_*(1-\delta_0)\,d_{T_a}(x_a,y_a)
\Bigr\}.&
\end{align}
Henceforth, we will simply write $\delta_0$ to mean $\delta_0(L,L')$.
\end{comment}

\vspace{0.3cm}
\noindent
\underline{Second family}: For $m,m'\in\mathcal{B}(M)$, define
\begin{align}\label{eq:ALL'mm'}
    A_{L,L'}^{m,m'}
    :=
    \Bigl\{
    a\in A_{L,L'}:
    &\ \exists\, f_a:T_a\to M
    \text{ as in } \eqref{eq:ALL'}, \nonumber\\
    &\ \exists\,\text{adjacent branch points }x_a,y_a\in H_a
    \text{ such that} \nonumber\\
    &\quad f_a(x_a)=m,\qquad f_a(y_a)=m', \nonumber\\
    &\quad d_M(f_a(x_a),f_a(y_a))
    >
    L(1-\delta_{L'})\,d_{T_a}(x_a,y_a)
    \Bigr\}.
\end{align}

Next, we prove that the set $A_{L,L'}$ can be covered by a family of sets $A_{L,L'}^{m,m'}$.
\begin{lemma}\label{lem:AL_cover}
    For every $L\in \Q_{>0},L'\in\Q_{\ge 1}$,
    \[
    A_{L,L'}
    \subseteq
    \bigcup_{m,m'\in\mathcal{B}(M)} A_{L,L'}^{m,m'}.
    \]
\end{lemma}

\begin{proof}
Let $a\in A_{L,L'}$. Choose an embedding $f_a:T_a\to M$ as in the definition of
$A_{L,L'}$, and let $\lambda_a = \Lip(f_a|_{H_a})$.
By definition of $A_{L,L'}$, $L(1-\delta_{L'})<\lambda_a\le L$.
Choose $\varepsilon>0$ so small that
\[
    L(1-\delta_{L'})<\lambda_a-\varepsilon.
\]
By Lemma~\ref{lem:adjacent_horizontal_pair}, there exist adjacent branch points
$x_a,y_a\in H_a$ such that
\[
    \frac{d_M(f_a(x_a),f_a(y_a))}{d_{T_a}(x_a,y_a)}
    >
    \lambda_a-\varepsilon
    >
    L(1-\delta_{L'})
\]
and hence 
\[
    d_M(f_a(x_a),f_a(y_a))\, >\, L(1-\delta_{L'}) d_{T_a}(x_a, y_a).
\]
Denoting $m:=f_a(x_a)$ and $m':=f_a(y_a)$, we have $m,m'\in \mathcal{B}_M$ by Lemma~\ref{lem:tree_embedding_arcs_heights}. Hence $a\in A_{L,L'}^{m,m'}$. Since $a\in A_{L,L'}$ was arbitrary, the claim follows.
\end{proof}

\begin{comment}
    
\begin{lemma}[Covering family]\label{lem:AL_cover}
For every $L,L'\ge 1$,
\[
A_{L,L'}
\subseteq
\bigcup_{L_*\in\Q\cap[1/L',L]}
\bigcup_{m,m'\in\mathcal{B}(M)}
A^{L,L'}_{L_*,m,m'}.
\]
\end{lemma}

\begin{proof}
Let $a\in A_{L,L'}$. Choose an embedding $f_a:T_a\to M$ as in the definition of
$A_{L,L'}$, and let
\[
\lambda_a:=\Lip_{\mathrm{hor}}(f_a).
\]
Then
\[
\frac{1}{L'}\le\lambda_a\le L.
\]
Choose a rational number $L_*\in\Q\cap[1/L',L]$ such that
\[
\lambda_a\le L_*<\frac{\lambda_a}{1-\delta_0}.
\]
Hence $L_*(1-\delta_0)<\lambda_a$. Choose $\varepsilon>0$ so small that
\[
L_*(1-\delta_0)<\lambda_a-\varepsilon.
\]
By Lemma~\ref{lem:adjacent_horizontal_pair}, there exist adjacent branch points
$x_a,y_a\in H_a$ such that
\[
\frac{d_M(f_a(x_a),f_a(y_a))}{d_{T_a}(x_a,y_a)}
>
\lambda_a-\varepsilon
>
L_*(1-\delta_0).
\]
Therefore
\[
d_M(f_a(x_a),f_a(y_a))
\ge
L_*(1-\delta_0)\,d_{T_a}(x_a,y_a).
\]
Let $m:=f_a(x_a)$ and $ m':=f_a(y_a)$. Since $x_a$ and $y_a$ are branch points of $T_a$ and $f_a$ is an embedding of
trees, we have $m,m'\in\mathcal{B}(M)$. Thus
\[
a\in A^{L,L'}_{L_*,m,m'}.
\]
Since $a\in A_{L,L'}$ was arbitrary, the desired covering follows.
\end{proof}
\end{comment}

\subsection{Proof of Theorem~\ref{thm:mainthm1}}
Let $M \in \mathcal{GT}(\bfn,\bfc)$ and define
\[ 
    E_M :=
        \{a\in[1/4,1/3]: T_a \text{ admits a bi-Lipschitz embedding into } M\}. 
\]

We prove that $E_M$ is countable. Suppose, to the contrary, that $E_M$ is uncountable. By the covering claim above, 
\[
    E_M
    \subseteq \bigcup_{\substack{
            L\in\Q_{>0},\,L'\in\Q_{\ge 1}\\
            L<L'}}
        A_{L,L'}.
\]
There exist $L\in\Q_{>0}$ and $L'\in\Q_{\ge1}$, with $L<L'$, such that $A_{L,L'}$ is uncountable. Furthermore, by Lemma~\ref{lem:AL_cover},
\begin{equation*}
   A_{L,L'}
    \subseteq
    \bigcup_{m,m'\in\mathcal B_{M}}
    A_{L,L'}^{m,m'}.
\end{equation*}

Since $M$ is doubling and has bounded valency, the set $\mathcal{B}(M)$ of branch points of $M$ is countable by Lemma~\ref{lem:countable_branch_points}. Thus there exits a family of parameters $A_* := A_{L,L'}^{m,m'}$ so that $A_*$ is uncountable. 

Let $\eta$ be as fixed in
\S~\ref{subsec:family_parameters}. For the fixed $L,L'$, set
\[
    \delta_1
    :=
    \frac{\delta_{L'}}{\eta}
    =
    \frac{\bfc c_M}{100(L')^2}.
\]
Since $\eta<1$, we have
\[
    0<\delta_{L'}<\delta_1.
\]
Moreover, since $\bfc,c_M<1$ and $L'\ge 1$,
\[
    \delta_1<\frac1{100}<\frac12.
\]
Since $L<L'$, we also have
\[
    \delta_1
    =
    \frac{\bfc c_M}{100(L')^2}
    <
    \frac{\bfc c_M}{24LL'},
\]
and thus,
\[
    0<\delta_{L'}<\delta_1
    <
    \min\left\{
    \frac12,
    \frac{\bfc c_M}{24LL'}
    \right\}
    =
    r_M(L,L').
\]
Since $A_*\subset[1/4,1/3]$ is uncountable, there exist
distinct $a,b\in A_*$ such that
\[
    |a-b|<\delta_{L'}<\delta_1.
\]

% Let $$ \delta_1 := \delta_{L'}/\eta.$$
% \begin{tikzpicture}
%     \draw[springgreen] (0,0) -- (12,0);
% \end{tikzpicture}
\begin{comment}   
For this fixed $L'$, let $\delta_0(L,L')=\delta_0(L',L') < \tfrac{c\tau}{6 (L')^2}$ be the constant given by Lemma~\ref{lem:maxstretch}.
\noindent 
Denote
    \[
    \delta_0:=\delta_0(L,L')
    \] 
By shrinking $\delta_0$ if necessary, we may assume
\[
\delta_0<\eta\,\frac{c^2}{24 \,(L')^2} < \eta\,\frac{c^2}{24 \,L L'} \qquad  \text{ since }L < L'.
\]
Since $A_*\subset [1/4,1/3]$ is uncountable, there exist distinct
$a,b\in A_*$ such that
\[
|a-b|<\delta_0.
\]
\end{comment}
\noindent
Thus, by definition of $A_*$, we may find embeddings
\[
    f_a:T_a\to M,
    \qquad
    f_b:T_b\to M,
\]
and adjacent branch points
\[
x_a,y_a\in H_a,
\qquad
x_b,y_b\in H_b
\]
for which 
\begin{equation}\label{eq:endpoints_match_new}
    f_a(x_a)=m=f_b(x_b),
    \qquad
    f_a(y_a)=m'=f_b(y_b),
\end{equation}
and
\[
d_{M}(f_a(x_a),f_a(y_a))
\ge
L(1-\delta_{L'})\,d_{T_a}(x_a,y_a),
\]
\[
d_{M}(f_b(x_b),f_b(y_b))
\ge
L(1-\delta_{L'})\,d_{T_b}(x_b,y_b).
\]
Write $J_a:=[x_a,y_a]\subset H_a$ and $J_b:=[x_b,y_b]\subset H_b$. Then $J_a$ and $J_b$ are $(L,\delta_{L'})$--well-stretched. Since $
|a-b|<\delta_{L'}$, all hypotheses of Lemma~\ref{lem:maxstretch} are satisfied. Hence, if
\[
    \{m_a^1,m_a^2\}
    =
    \mathcal I_1^a(J_a)\setminus\{x_a,y_a\},
    \qquad
    \{m_b^1,m_b^2\}
    =
    \mathcal I_1^b(J_b)\setminus\{x_b,y_b\},
\]
then
\[
f_a(m_a^i)=f_b(m_b^i),
\qquad i=1,2.
\]
In other words, the first-generation subdivision points of $J_a$ and $J_b$
have the same images in $ M$.

Now by the choice of $\delta_1$,  
% \[
% \delta_1 :=\delta_{L'}/\eta
% \frac{\delta_0}{\eta}<\delta_1<\frac{c^2}{24\,L_*L'}.
% \]
% \[
% \delta_1 :=\frac{\delta_{L'}}{\eta} > \delta_{L'}.
% \]
since $\delta_{L'}< \delta_1$, then $1-\delta_{L'}> 1-\delta_1$, thus every $(L,\delta_{L'})$--well-stretched interval is also $(L,\delta_1)$--well-stretched.  Moreover, $|a-b|< \delta_{L'}< \delta_1$.

We apply Lemma~\ref{lem:shrinkingsmall} to $f_a|_{J_a}$ and $f_b|_{J_b}$
with parameters $\eta$ and $\delta_1$, taking
\[
    \varepsilon=\eta\delta_1=\delta_{L'}.
\] 
Indeed, the $(L,\delta_{L'})$--well-stretching of $J_a$ and
$J_b$ gives precisely the endpoint-stretching hypothesis required
by that lemma and we obtain bad sets
\[
    B_a\subset J_a,
    \qquad
    B_b\subset J_b
\]
such that
\[
    \mathcal L(B_a)\le \eta |J_a|,
    \qquad
    \mathcal L(B_b)\le \eta |J_b|.
\]
Moreover, if $p_a\in J_a\setminus B_a$, then every descendant interval of $J_a$ containing $p_a$ is
$(L,\delta_1)$-well-stretched. The analogous statement holds for every $p_b\in J_b\setminus B_b$.

In the arguments below, we always use Lemma~\ref{lem:shrinkingsmall} in the following
pointwise sense: if a point $p\in J_a\setminus B_a$, then every interval along the
combinatorial address of $p$ is $(L,\delta_1)$--well-stretched; similarly for
points in $J_b\setminus B_b$.

% Let
% \[
% \varepsilon:=\eta\delta_1.
% \]
% Then $\delta_0<\varepsilon<\delta_1$. Hence every $(L,\delta_0)$-well-stretched interval is also
% $(L,\varepsilon)$-well-stretched and $(L,\delta_1)$-well-stretched. The significance of this choice is that a parent interval coming from the definition of one of the families $A_{L,L'}^{m,m'}$ automatically satisfies the hypotheses of both Lemma~\ref{lem:maxstretch} and Lemma~\ref{lem:shrinkingsmall}.

Applying Lemma~\ref{lem:maxstretch} to the parent intervals $J_a$ and $J_b$, and using
\eqref{eq:endpoints_match_new}, we conclude that the first-generation subdivision points
have the same images. In particular, if
\[
    \{m_a^1,m_a^2\}=\mathcal I_1^a(J_a)\setminus\{x_a,y_a\},
    \qquad
    \{m_b^1,m_b^2\}=\mathcal I_1^b(J_b)\setminus\{x_b,y_b\},
\]
then
\[
    f_a(m_a^i)=f_b(m_b^i),
    \qquad i=1,2.
\]
For each $n\ge 0$ and each address $\omega\in\{1,2,3\}^n$, let
\[
    I_\omega^a\in\mathcal D_n^a(J_a),
    \qquad
    I_\omega^b\in\mathcal D_n^b(J_b)
\]
denote the corresponding level-$n$ subintervals determined by the subdivision rule.

\noindent
\textit{Claim 1 (combinatorial addresses).}
Let $\omega\in\{1,2,3\}^n$. Suppose that every ancestor of $I_\omega^a$ and of
$I_\omega^b$ is $(L,\delta_1)$--well-stretched with the convention that every ancestor includes the interval itself. Then
\[
    f_a(s_{I_\omega^a})=f_b(s_{I_\omega^b})
    \qquad\text{and}\qquad
    f_a(t_{I_\omega^a})=f_b(t_{I_\omega^b}).
\]

\begin{proof}
    We argue by induction on $n$.
    For $n=0$, we have
    \[
        I_\emptyset^a=J_a,
        \qquad
        I_\emptyset^b=J_b,
    \]
    so the endpoint agreement follows from the definition of $A_*$.
    
    Assume the statement holds at level $n$, and let $\omega\in\{1,2,3\}^n$ be such that
    all ancestors of $I_\omega^a$ and $I_\omega^b$ are $(L,\delta_1)$--well-stretched.
    By the induction hypothesis, the endpoints of $I_\omega^a$ and $I_\omega^b$ have identical images in $M$. Since $|a-b|<\delta_{L'}<\delta_1$, Lemma~\ref{lem:maxstretch}
    applies to the pair $(I_\omega^a,I_\omega^b)$ and shows that their first-generation
    subdivision points have identical images. Hence each child of $I_\omega^a$ has the
    same endpoint images as the corresponding child of $I_\omega^b$, proving the induction step.
\end{proof}
\vspace{0.2cm}

Now fix $\tau>0$ so small that
\[
    [\alpha_a-\tau,\alpha_a+\tau]
    \cap
    [\alpha_b-\tau,\alpha_b+\tau]
    =
    \emptyset.
\]
After interchanging $a$ and $b$ if necessary, assume that
\[
    \alpha_a+\tau<\alpha_b-\tau.
\]
Choose $n\in\mathbb N$ sufficiently large so that
Lemma~\ref{lem:lawoflargenum} applies to both
$\mathcal D_n^a(J_a)$ and $\mathcal D_n^b(J_b)$ with parameters
$(\tau,\eta)$ and, in addition,
\begin{equation}\label{eq:choice_n_final}
(L')^4e^{-n(\alpha_b-\alpha_a-2\tau)}<1.
\end{equation}

Let
\[
\Adm_{n,a}\subseteq\mathcal D_n^a(J_a),
\qquad
\Adm_{n,b}\subseteq\mathcal D_n^b(J_b)
\]
denote the collections of admissible level-$n$ intervals, where
admissibility is understood with $\delta'=\delta_1$ in Definition~\ref{def:admissible}.

Let
\[
     N_{n,a}
    :=
    \bigcup
    \left\{
    I\in\mathcal D_n^a(J_a):
    I\text{ is non-typical}
    \right\},
\]
and define $ N_{n,b}$ analogously. By
Lemma~\ref{lem:lawoflargenum},
\[
    \mathcal L( N_{n,a})
    \le
    \eta|J_a|,
    \qquad
    \mathcal L( N_{n,b})
    \le
    \eta|J_b|.
\]

Define
\[
    \mathcal A_a
    :=
    \left(
    (J_a\setminus B_a)
    \cap
    \bigcup_{I\in\Adm_{n,a}}I
    \right)
    \setminus
    \mathcal I_n^a(J_a),
\]
and,
\[
    \mathcal A_b
    :=
    \left(
    (J_b\setminus B_b)
    \cap
    \bigcup_{I\in\Adm_{n,b}}I
    \right)
    \setminus
    \mathcal I_n^b(J_b).
\]
We claim that
\[
    J_a\setminus\mathcal A_a
    \subseteq
    B_a\cup N_{n,a}\cup\mathcal I_n^a(J_a).
\]
Indeed, suppose that
\[
    p\in
    J_a\setminus
    \left(
    B_a\cup N_{n,a}\cup\mathcal I_n^a(J_a)
    \right).
\]
Since $p\notin\mathcal I_n^a(J_a)$, it lies in a unique interval
$I\in\mathcal D_n^a(J_a)$. Since $p\notin N_{n,a}$, the
interval $I$ is typical. Since $p\notin B_a$, every
interval along the combinatorial address of $p$, and, in paritcular
$I$, is $(L,\delta_1)$--well-stretched. Thus,
$I\in\Adm_{n,a}$ and hence $p\in\mathcal A_a$. This proves the
claim. Similarly,
\[
J_b\setminus\mathcal A_b
\subseteq
B_b\cup N_{n,b}\cup\mathcal I_n^b(J_b).
\]

Since $\mathcal I_n^a(J_a)$ and $\mathcal I_n^b(J_b)$ are finite, from
Lemma~\ref{lem:shrinkingsmall} and
Lemma~\ref{lem:lawoflargenum}, we obtain
\[
\mathcal L(J_a\setminus\mathcal A_a)
\le
2\eta|J_a|,
\qquad
\mathcal L(J_b\setminus\mathcal A_b)
\le
2\eta|J_b|.
\]
Equivalently,
\[
\mathcal L(\mathcal A_a)
\ge 
(1-2\eta)|J_a|,
\qquad
\mathcal L(\mathcal A_b)
\ge 
(1-2\eta)|J_b|.
\]

Since $f_a|_{H_a}$ and $f_b|_{H_b}$ are $L$-Lipschitz,
Lemma~\ref{lem:hmeausre_lispchitz_maps} gives
\[
    \mathcal H^1\bigl(f_a(J_a\setminus\mathcal A_a)\bigr)
    \le
    L\mathcal L(J_a\setminus\mathcal A_a)
    \le
    2L\eta|J_a|,
\]
\[
    \mathcal H^1\bigl(f_b(J_b\setminus\mathcal A_b)\bigr)
    \le
    L\mathcal L(J_b\setminus\mathcal A_b)
    \le
    2L\eta|J_b|.
\]
Since $J_a$ and $J_b$ are $(L,\delta_{L'})$--well-stretched
and their ordered endpoint images are $m,m'$, we have
\[
    d_M(m,m')
    >
    L(1-\delta_{L'})|J_a|,
\]
\[
    d_M(m,m')
    >
    L(1-\delta_{L'})|J_b|.
\]
Thus,
\[
    \mathcal H^1\bigl(f_a(J_a\setminus\mathcal A_a)\bigr)
    <
    \frac{2\eta}{1-\delta_{L'}}\,d_M(m,m'),
\]
\[
    \mathcal H^1\bigl(f_b(J_b\setminus\mathcal A_b)\bigr)
    <
    \frac{2\eta}{1-\delta_{L'}}\,d_M(m,m').
\]
Therefore, since $\delta_{L'} <1/2$, and $\eta<1/100$, we obtain
\begin{align*}
    &\mathcal H^1\left(
    f_a(J_a\setminus\mathcal A_a)
    \cup
    f_b(J_b\setminus\mathcal A_b)
    \right)\\
    &\qquad<
    \frac{4\eta}{1-\delta_{L'}}\,d_M(m,m')
    <
    8\eta\,d_M(m,m')
    <
    d_M(m,m').
\end{align*}

Since $f_a|_{J_a}$ and $f_b|_{J_b}$ are embeddings of intervals
into the metric tree $M$, their images are the unique arcs joining
their endpoint images. Hence
\[
    f_a(J_a)=[m,m']_M=f_b(J_b).
\]
In addition, we have $ \mathcal H^1([m,m']_M)=d_M(m,m')$. It follows that there exists
\[
    z\in[m,m']_M
    \setminus
    \left(
    f_a(J_a\setminus\mathcal A_a)
    \cup
    f_b(J_b\setminus\mathcal A_b)
    \right).
\]

Let
\[
p_a:=(f_a|_{J_a})^{-1}(z),
\qquad
p_b:=(f_b|_{J_b})^{-1}(z).
\]
By the choice of $z$, $p_a\in\mathcal A_a$ and $p_b\in\mathcal A_b$. In particular,
\[
    p_a\notin B_a,
    \quad
    p_b\notin B_b,
    \qquad \text{ and } \qquad 
    p_a\notin\mathcal I_n^a(J_a), 
    \quad
    p_b\notin\mathcal I_n^b(J_b).
\]
Thus there exist unique words
\[
\omega^a,\omega^b\in\{1,2,3\}^n
\]
such that
\[
p_a\in I_{\omega^a}^a,
\qquad
p_b\in I_{\omega^b}^b.
\]
Because $p_a\in\mathcal A_a$ and $p_b\in\mathcal A_b$, these
unique intervals satisfy
\[
I_{\omega^a}^a\in\Adm_{n,a},
\qquad
I_{\omega^b}^b\in\Adm_{n,b}.
\]

\noindent
\textit{Claim 2.} We claim that $\omega^a=\omega^b$.
\begin{proof}
Let $k$ be the length of the longest common prefix of $\omega^a$ and
$\omega^b$. Suppose, toward a contradiction, that $k<n$, and denote
this common prefix by $\sigma$. Let
\[
K_a:=I_\sigma^a,
\qquad
K_b:=I_\sigma^b.
\]
Since $p_a\notin B_a$ and $p_b\notin B_b$, every interval along their
respective combinatorial addresses is $(L,\delta_1)$--well-stretched.
Therefore, by Claim~1
\[
f_a(s_{K_a})=f_b(s_{K_b}),
\qquad
f_a(t_{K_a})=f_b(t_{K_b}).
\]
Lemma~\ref{lem:maxstretch} now shows that the corresponding subdivision
points of $K_a$ and $K_b$ have the same images. Consequently, for each
$i\in\{1,2,3\}$, the $i$-th child of $K_a$ and the $i$-th child of
$K_b$ have the same image in $M$. In addition,
\[
    p_a\notin\mathcal I_n^a(J_a),
    \qquad
    p_b\notin\mathcal I_n^b(J_b),
\]
so neither $p_a$ nor $p_b$ is a subdivision point at level $k+1$.
Thus, by injectivity of $f_a$ and $f_b$, the point
\[
z=f_a(p_a)=f_b(p_b)
\]
lies in the relative interior of one of these child arcs. Since
corresponding children have the same images, $p_a$ and $p_b$ lie in
corresponding children. Hence the $(k+1)$-st symbols of $\omega^a$ and
$\omega^b$ agree, contradicting the maximality of $k$. Hence $\omega^a=\omega^b$.
\end{proof}

Applying Claim~1 once more, now to the common address $\omega:=\omega^a=\omega^b$, gives
\[
f_a(s_{I_\omega^a})=f_b(s_{I_\omega^b}),
\qquad
f_a(t_{I_\omega^a})=f_b(t_{I_\omega^b}).
\]
In particular,
\[
    d_M\bigl(
    f_a(s_{I_\omega^a}),
    f_a(t_{I_\omega^a})
    \bigr)
    =
    d_M\bigl(
    f_b(s_{I_\omega^b}),
    f_b(t_{I_\omega^b})
    \bigr).
\]
The global $(1/L',L')$-bi-Lipschitz bounds therefore give
\begin{equation}\label{eq:ratio_admissible_new}
    \frac{1}{(L')^2}
    \le
    \frac{|I_\omega^a|}{|I_\omega^b|}
    \le
    (L')^2.
\end{equation}
Similarly, since $J_a$ and $J_b$ have the same endpoint images,
\begin{equation}\label{eq:ratio_parent_new}
    \frac{1}{(L')^2}
    \le
    \frac{|J_a|}{|J_b|}
    \le
    (L')^2.
\end{equation}

Since $I_\omega^a\in\Adm_{n,a}$ and
$I_\omega^b\in\Adm_{n,b}$, their typicality gives
\[
    e^{(\alpha_a-\tau)n}
    \le
    \frac{|I_\omega^a|}{|J_a|}
    \le
    e^{(\alpha_a+\tau)n}
\]
and
\[
    e^{(\alpha_b-\tau)n}
    \le
    \frac{|I_\omega^b|}{|J_b|}
    \le
    e^{(\alpha_b+\tau)n}.
\]
Using \eqref{eq:ratio_parent_new}, we obtain
\begin{align*}   
    \frac{|I_\omega^a|}{|I_\omega^b|}
    &=
    \frac{|I_\omega^a|}{|J_a|}
    \frac{|J_a|}{|J_b|}
    \frac{|J_b|}{|I_\omega^b|}\\
    &\le
    e^{(\alpha_a+\tau)n}
    (L')^2
    e^{-(\alpha_b-\tau)n}\\
    &=
    (L')^2e^{-n(\alpha_b-\alpha_a-2\tau)}.
\end{align*}
By \eqref{eq:choice_n_final}, we obtain 
\[
    \frac{|I_\omega^a|}{|I_\omega^b|}
    <
    \frac{1}{(L')^2},
\]
contradicting \eqref{eq:ratio_admissible_new}.

This contradiction disproves the assumption that $E_M$ is
uncountable. Hence $E_M$ is countable, as required.
\qed

\section{Bi-Lipschitz embedding of ultrametric spaces}\label{S:ultrametric}

The proof of the failure of a universal element among trees was based on connectivity. Motivated by this connection, we consider in this section the analogous universality problem for ultrametric spaces, which are totally disconnected. In contrast with the negative result for geodesic trees, we obtain a positive result when the Assouad dimension of the source is strictly smaller than the lower Assouad dimension of the target.

The proof is again tree-theoretic. The Assouad dimension of the source
controls the number of children in its associated tree, while the lower
Assouad dimension of the target guarantees enough separated sub-balls
to realize those children. Passing to limits along branches then
produces the required embedding.

\subsubsection{Relation to previous work.}
Bi-Lipschitz embeddings of ultrametric spaces have been studied for
several types of targets. Luukkainen and Movahedi-Lankarani proved
that an ultrametric space $X$ embeds bi-Lipschitzly into $\R^n$
whenever $\adim X<n$, and Luosto proved that this condition is also necessary
\cite{LuukLank94:UMBL,Luosto1996}. Bonk and Foertsch obtained embeddings of compact doubling ultrametric spaces into suitable symbolic Cantor spaces
\cite[Proposition~6.3]{BonkFoertsch06:asymp_curv_um}, while Brodskiy, Dydak, Higes, and Mitra constructed a fixed complete separable ultrametric space that is $3$-bi-Lipschitz universal for all separable ultrametric spaces
\cite[Corollary~2.8]{BrodDydaetal07:nagata_dim_0}.

Embedding results for regular targets were obtained by Mattila and Saaranen and, more generally, by Lü, Lou, Wen, and Xi
\cite[Theorem~3.3]{MatillaSaaranen09:AR_BL}
\cite[Theorem~2]{FMZL15:blhomofractalUM}. In particular, the latter result includes the case of a compact $s$-Ahlfors regular ultrametric source and a compact $t$-Ahlfors regular target when $s<t$.

The point of Theorem~\ref{thm:um_embedding_intro} is not the Ahlfors-regular specialization itself. Rather, it replaces regularity and homogeneity by one-sided covering estimates: the source is controlled only by its Assouad dimension, and the target only by its lower Assouad dimension. Thus the source may be nonhomogeneous, and the prescribed complete target need not be Ahlfors regular, doubling, homogeneous, or ultrametric.

\subsection{Preliminaries}\label{ss:prelim_um} In this subsection we introduce notions of ultrametric spaces, Ahlfors regularity and Lower Assouad dimension needed for the proof of Theorem~\ref{thm:um_embedding_intro}.

Recall that for a nonempty subset $E\subseteq X$ and $r>0$, $N_X(E, r)$ denotes the smallest number of balls of radius $r$ required
to cover $E$. If no such finite cover exists, we set $N_X(E, r)=\infty$. When $X$ is understood from the context, we simply write $N(E,r)$.

\begin{definition}[Ahlfors regularity]
Let $Q>0$. A metric space $(X,d_X)$ is called
$Q$-Ahlfors regular if there exist a Borel measure $\mu$ on $X$
and a constant $C\ge 1$ such that
\[
    C^{-1}r^Q
    \le 
    \mu\bigl(B_X(x,r)\bigr)
    \le 
    Cr^Q
\]
for every $x\in X$ and every $0<r\le \diam(X)$. If $X$ is unbounded,
the inequalities are required to hold for every $r>0$.
\end{definition}

\begin{definition}[Lower Assouad dimension]
The lower Assouad dimension of a metric space $X$ is
\[
\ladim X
=
\sup\left\{
q\ge 0:
\begin{array}{l}
\text{there exists $c_X>0$ such that, for every $x\in X$}\\[1mm]
\text{and every $0<r<R<\diam(X)$,}\\[1mm]
N_X\bigl(B_X(x,R), r\bigr)
\ge 
c_X\left(\dfrac{R}{r}\right)^q
\end{array}
\right\}.
\]
When $X$ is unbounded, the condition $R<\diam(X)$ is omitted.
\end{definition}

Thus, the lower Assouad dimension measures the smallest local
covering growth occurring anywhere in the space and at any pair of
scales. In contrast, the Assouad dimension measures the largest such
growth.

If $X$ is $Q$-Ahlfors regular, then
\[
\ladim X
=
\dim_{\mathrm A}X
=
Q.
\]
In particular, every $Q$-Ahlfors regular target has lower Assouad
dimension $Q$. For more details, see \cite{fraser:ass}.

\begin{definition}[Ultrametric space]
A metric space $(X,d_X)$ is called an ultrametric space if
\[
    d_X(x,z)
    \le
    \max\{d_X(x,y),d_X(y,z)\}
\]
for every $x,y,z\in X$. This inequality is called the strong triangle inequality.
\end{definition}

Ultrametric balls have a hierarchical structure. Every point of an
ultrametric ball is a center of that ball, and any two balls that
intersect are nested. More precisely, if
\[
B_X(x,r)\cap B_X(y,s)\neq\varnothing
\qquad\text{and}\qquad
r\le s,
\]
then
\[
B_X(x,r)\subseteq B_X(y,s).
\]
In particular, two balls of the same radius are either equal or
disjoint.

Ultrametric spaces are naturally related to rooted trees through their
hierarchies of nested balls; see \cite{Hughes04:UM}. In the proof
of Theorem~\ref{thm:um_embedding_intro}, we encode the ultrametric space by a discrete tree of balls and realize this tree inside the target by a separated family of nested balls.

Let $X$ be bounded, and fix a sequence
\[
    r_0>r_1>r_2>\cdots
\]
converging to $0$, with $r_0\geq\diam(X)$. For each $n\geq0$, define
\[
    x\sim_n y
    \quad\Longleftrightarrow\quad
    d_X(x,y)\le  r_n.
\]
The strong triangle inequality implies that $\sim_n$ is an
equivalence relation. The equivalence class of $x$ is precisely
\[
    [x]_n=\overline B_X(x,r_n).
\]

A vertex at level $n$ is an equivalence class in $\mathcal P_n$,
where $\mathcal P_n$ denotes the partition induced by $\sim_n$.
Since $r_0\geq\diam(X)$, we have
\[
    \mathcal P_0=\{X\},
\]
so $X$ is the root. A vertex in $\mathcal P_{n+1}$ is joined to the
unique vertex in $\mathcal P_n$ containing it.

Thus, moving down the tree corresponds to passing to successively
smaller nested balls. If $x,y\in X$ belong to the same class at level
$n$ but to distinct classes at level $n+1$, then
\[
    r_{n+1}<d_X(x,y)\le  r_n.
\]
Hence the distance between two points is determined, up to the ratio
$r_n/r_{n+1}$, by the first level at which their corresponding
branches separate. In particular, for geometric scales
$r_n=\delta^n$ for small $\delta>0$, the first splitting level determines the distance up
to the fixed multiplicative factor $\delta^{-1}$.

Every point $x\in X$ determines a descending branch
\[
    [x]_0\supseteq[x]_1\supseteq[x]_2\supseteq\cdots
\]
of the ball tree. If $X$ is complete, then every descending branch
whose radii tend to zero determines a unique point of $X$.

% \begin{remark} %
% After rescaling, suppose that $\diam(X)\le 1$. There is also a
% canonical rooted $\mathbb R$-tree associated with $X$. Define
% \[
%     T_X=(X\times[0,\infty))/\!\sim,
% \]
% where
% \[
%     (x,s)\sim(y,t)
%     \quad\Longleftrightarrow\quad
%     s=t
%     \ \text{ and }\
%     d_X(x,y)\leq e^{-t}.
% \]
% Writing $[x,t]$ for the equivalence class of $(x,t)$, the tree metric
% is
% \[
%     d_{T_X}\bigl([x,s],[y,t]\bigr)
%     =
%     s+t
%     -
%     2\min\{s,t,-\log d_X(x,y)\},
% \]
% with the convention $-\log 0=\infty$.

% The root is $o=[x,0]$, which is independent of $x$ because
% $\diam(X)\leq1$. Each $x\in X$ determines a geodesic ray
% \[
%     \xi_x(t)=[x,t].
% \]
% For $x\neq y$, the rays $\xi_x$ and $\xi_y$ agree up to height
% $-\log d_X(x,y)$ and then separate. Consequently,
% \[
%     (\xi_x\mid\xi_y)_o=-\log d_X(x,y),
% \]
% and the natural end metric satisfies
% \[
%     d_{\partial T_X}(\xi_x,\xi_y)
%     =
%     \exp\bigl(-(\xi_x\mid\xi_y)_o\bigr)
%     =
%     d_X(x,y).
% \]
% If $X$ is complete, then $X$ is isometric to the full end space of
% $T_X$. If $X$ is not complete, the full end space is naturally
% isometric to the completion of $X$. See \cite{Hughes04:UM}.
% \end{remark}

\subsection{Consequences of Theorem~\ref{thm:um_embedding_intro}}
% This is the weakest assumption I can think about

Before presenting the proof, we record several consequences of
Theorem~\ref{thm:um_embedding_intro}.

\begin{corollary} 
Let $X$ be a bounded ultrametric space and let $Y$ be a complete
$Q$-Ahlfors regular metric space. If
\[
    \adim X<Q,
\]
then $X$ admits a bi-Lipschitz embedding into $Y$.
\end{corollary}

\begin{proof}
Since $Y$ is $Q$-Ahlfors regular, we have $\ladim Y=Q$. The conclusion
follows from Theorem~\ref{thm:um_embedding_intro}.
\end{proof}

Equivalently, every complete $Q$-Ahlfors regular space is
bi-Lipschitz universal for bounded ultrametric spaces of Assouad
dimension strictly less than $Q$.

Since $\mathbb R^n$ is $n$-Ahlfors regular, we recover the bounded
case of the theorem of Luukkainen and Movahedi-Lankarani \cite{LuukLank94:UMBL}. 

% \begin{corollary}
% Let $X$ be a bounded ultrametric space and let $n\ge 1$. Then $X$
% admits a bi-Lipschitz embedding into $\mathbb R^n$ if and only if
% \[
% \adim X<n.
% \]
% \end{corollary}

\begin{corollary}
Let $X$ be a bounded uniformly disconnected metric space and let $Y$
be complete. If
\[
    \adim X<\ladim Y,
\]
then $X$ admits a bi-Lipschitz embedding into $Y$.
\end{corollary}

\begin{proof}
By the ultrametrization theorem for uniformly disconnected spaces \cite[Proposition~15.7]{DavidSemmes97:uniform_discon}, there is an ultrametric $\rho$ on $X$ such that $(X,d_X)$ and $(X,\rho)$ are bi-Lipschitz equivalent. Since Assouad dimension is
bi-Lipschitz invariant,
\[
    \adim(X,\rho)=\adim(X,d_X).
\]
Apply Theorem~\ref{thm:um_embedding_intro} to $(X,\rho)$ and then compose
with the bi-Lipschitz equivalence.
\end{proof}

% This also re
% \begin{corollary}
% Let $0<s<t$. Suppose that $X$ is a bounded, uniformly disconnected,
% $s$-Ahlfors regular space and that $Y$ is a complete
% $t$-Ahlfors regular space. Then $X$ admits a bi-Lipschitz embedding
% into $Y$.
% \end{corollary}

% \begin{proof}
% Ahlfors regularity gives
% \[
% \adim X=s
% \qquad\text{and}\qquad
% \ladim Y=t.
% \]
% The preceding corollary therefore applies.
% \end{proof}

For compact sources and targets, this recovers the Ahlfors regular Theorem~2 and Remark~6 of L\"u-Lou-Wen-Xi \cite{FMZL15:blhomofractalUM}.

There are a few remarks in order. 

\begin{remark}
    It is enough that $Y$ contain a complete subspace $Z$ satisfying
    \[
        \ladim Z>\adim X.
    \]
    Indeed, Theorem~\ref{thm:um_embedding_intro} gives a bi-Lipschitz embedding
    of $X$ into $Z$, and hence into $Y$. This formulation is useful because lower Assouad dimension is not monotone under inclusion.
\end{remark}

If one wishes to allow an arbitrary, not necessarily complete,
subspace $Z\subseteq Y$, one may pass to its closure, using the fact
that lower Assouad dimension is unchanged under completion.

\begin{remark}
However, completeness of the target cannot be omitted. Let
\[
X=\{0,1\}^{\mathbb N},
\qquad
d(\omega,\tau)=3^{-|\omega\wedge\tau|},
\]
where $|\omega \wedge \tau|$ denotes the lenth of their maximal common prefix and let
\[
Y=\mathbb Q\cap[0,1].
\]
Then
\[
\adim X=\frac{\log 2}{\log 3}<1=\ladim Y,
\]
but no embedding $X\to Y$ exists because $X$ is uncountable and $Y$
is countable.
\end{remark}

\begin{remark}
The strict inequality in Theorem~\ref{thm:um_embedding_intro} cannot in
general be replaced by a non-strict inequality. For $n\ge 1$, let
\[
X_n=\{0,1\}^{\mathbb N}
\]
and define
\[
d_n(\omega,\tau)
=
2^{-|\omega\wedge\tau|/n}
\]
for $\omega\neq\tau$, where $|\omega\wedge\tau|$ denotes the length
of their maximal common prefix. Then $X_n$ is a compact
$n$-Ahlfors regular ultrametric space, and hence
\[
\adim X_n=n=\ladim\mathbb R^n.
\]
Nevertheless, \cite[Theorem 4.5]{Luosto1996} implies that $X_n$ does not admit a
bi-Lipschitz embedding into $\mathbb R^n$. Thus equality of the two dimensions does not provide a general embedding theorem. It may still permit an embedding for particular pairs of spaces.
\end{remark}

\subsection{Proof of Theorem~\ref{thm:um_embedding_intro}}
If $X$ contains at most one point, then the conclusion is immediate. We may therefore assume that $0<\diam(X) < \infty$. Replacing $d_X$ by the rescaled metric $d_X/\diam(X)$, which is a bi-Lipschitz map, we may assume that $\diam(X) = 1$. Indeed, the rescaling is a similarity, and hence bi-Lipschitz. Similarly, we may assume that $\diam(Y) >1$.

Choose exponents $s$ and $t$ such that
\[
    \adim X<s<t<\ladim Y.
\]
After a further similarity rescaling of $Y$ if necessary, by the definitions of the Assouad and lower Assouad dimensions, there
exist constants $C_X\ge 1$, and $c_Y>0$ such that

\begin{equation}\label{eq:um_assouad-bound}
    N_X\bigl(B_X(x,R), r\bigr)
    \le 
    C_X\left(\frac{R}{r}\right)^s
\end{equation}
for every $x\in X$ and $0<r<R\le 1$, and
\begin{equation}\label{eq:um-lower-assoud-bound}
    N_Y\bigl(B_Y(y,R), r\bigr)
    \ge 
    c_Y\left(\frac{R}{r}\right)^t
\end{equation}
for every $y\in Y$ and $0<r<R\le 1$.

Since $s<t$, we may choose $\delta >0$ sufficiently small that
\begin{equation}\label{eq:delta_choice}
    0\,<\,\delta\,<\,\frac{1}{8} 
    \qquad \text{and} \qquad 
    c_Y(8\delta)^{-t} \,\ge\,C_X \delta^{-s}.
\end{equation}
Define the same source and target scales by 
\[
r_n = \delta^n,  \qquad n\ge 0.
\]

For every $n\ge 0$, define an equivalence relation on $X$ by
\[
    x\sim_n x'
    \quad\Longleftrightarrow\quad
    d_X(x,x')\le  r_n,
\]
and let $\mathcal P_n$ be the resulting partition. Since $r_0 =1$, we have $\mathcal{P}_0 = \{X\}$.

If $B\in\mathcal P_n$, define its collection of children by
\[
\operatorname{Ch}(B)
:=
\left\{
B'\in\mathcal P_{n+1}:B'\subseteq B
\right\}.
\]
A ball of radius $r_{n+1}$ can meet at most one child of $B$. Indeed,
two balls of the same radius in an ultrametric space are either equal
or disjoint. Consequently, \eqref{eq:um_assouad-bound} gives
\begin{equation}\label{eq:number-children}
    \#\operatorname{Ch}(B)
    \le 
    N_X(B, {r_{n+1}})
    \le 
    C_X\left(\frac{r_n}{r_{n+1}}\right)^s
    =
    C_X\delta^{-s}.
\end{equation}
Thus we obtain a tree of balls corresponding to $X$, and every vertex in the ball tree of $X$ has at most $C_X\delta^{-s}$ children.

We now construct a separated tree of nested balls inside $Y$. Pick any $y_{\varnothing}\in Y$. Assign
the ball
\[
    \mathcal B_{\varnothing}
    :=
    \overline B_Y(y_{\varnothing},1)
\]
to the root $\mathcal{P}_0$. Suppose that a target ball
\[
\mathcal B_v=\overline B_Y(y_v,r_n)
\]
has been assigned to a vertex $v$ at level $n$. The vertex $v$ corresponds to a ball, $V \subset X$ of radius $r_n$. Set \[
\operatorname{Ch}(V) := \{V_1, \ldots, V_m\} \quad\text{ where }\quad m\le C_X \delta^{-s}.
\]

We want to assign one smaller target ball to each child. We do not need to divide or cover all of $\mathcal B_v$. We only need to find at least $m$ mutually separated smaller balls inside it.

Since $$4 r_{n+1} = 4 \delta \delta^n < \frac{\delta^n}{2} = \frac{r_n}{2},$$
the lower Assouad estimate \eqref{eq:um-lower-assoud-bound} and \eqref{eq:delta_choice} give
\begin{align*}
    N\left(B_Y\left(y_v,\frac{r_n}{2}\right), 4 r_{n+1}\right) \,&\,\ge c_Y\left(\frac{r_n}{8r_{n+1}}\right)^t = c_Y (8\delta)^{-t} \\
    \,&\,\ge C_X \delta^{-s} \\
    \,&\,\ge \#\operatorname{Ch} (V).
\end{align*}

Choose a maximal $4r_{n+1}$-separated set $Z_v\subset B_Y\left(y_v,\frac{r_n}{2}\right)$.
By maximality, the balls of radius $4r_{n+1}$ centered at the points
of $Z_v$ cover $B_Y(y_v,r_n/2)$. Therefore,
\[
    \# Z_v
    \ge 
    N\left(B_Y\left(y_v,\frac{r_n}{2}\right),4r_{n+1}\right)
    \ge   \#\operatorname{Ch} (V).
\]
Choose distinct points
\[
    y_{v_1},\ldots,y_{v_m}\in Z_v
\]
corresponding to $V_1,\ldots,V_m$, and define
\[
    \mathcal B_{v_i}
    :=
    \overline B_Y(y_{v_i},r_{n+1}),
    \qquad i=1,\ldots,m.
\]
Then $\mathcal B_{v_i}\subseteq\mathcal B_v$,
since
\[
    d_Y(y_{v_i},y_v)+r_{n+1}
    <
    \frac{r_n}{2}+r_{n+1}
    =
    \left(\frac12+\delta\right)r_n
    <r_n.
\]
Moreover, for $i\neq j$,
\[
    \dist
    \bigl(\mathcal B_{v_i},\mathcal B_{v_j}\bigr)
    \ge 
    d_Y(y_{v_i},y_{v_j})-2r_{n+1}
    >
    2r_{n+1}.
\]

Repeating this construction inductively assigns a target ball
$\mathcal B_v$ to every vertex $v$ of the source tree, in such a way
that child balls are contained in their parent balls and distinct
children are uniformly separated.

For each $x\in X$, let $V_n(x)\in\mathcal P_n$ be the unique ball
containing $x$, and let $v_n(x)$ be its corresponding vertex. Then
\[
    \mathcal B_{v_0(x)}
    \supseteq
    \mathcal B_{v_1(x)}
    \supseteq
    \mathcal B_{v_2(x)}
    \supseteq\cdots .
\]
If $m\geq n$, then both $y_{v_m(x)}$ and $y_{v_n(x)}$ belong to
$\mathcal B_{v_n(x)}$. Consequently,
\[
    d_Y\bigl(y_{v_m(x)},y_{v_n(x)}\bigr)
    \le 
    \diam\bigl(\mathcal B_{v_n(x)}\bigr)
    \le 2r_n.
\]
Since $r_n=\delta^n\to0$, the sequence
$\bigl(y_{v_n(x)}\bigr)_{n\ge 0}$ is Cauchy. By completeness of $Y$,
we may define
\[
    f(x):=\lim_{n\to\infty}y_{v_n(x)}.
\]

% For each fixed $n$, the tail of this sequence is contained in the
% closed ball $\mathcal B_{v_n(x)}$. Therefore,
% \[
%     f(x)\in\mathcal B_{v_n(x)}
%     \qquad\text{for every }n\ge 0.
% \]
In particular,
\[
    \{f(x)\}
    =
    \bigcap_{n=0}^{\infty}\mathcal B_{v_n(x)},
\]
because the diameters of these nested balls tend to zero.

It remains to prove that $f$ is bi-Lipschitz. Let $x,x'\in X$ be
distinct, and let $n$ be the last level at which they belong to the
same element of $\mathcal P_n$. Thus,
\[
    V_n(x)=V_n(x'),
    \qquad
    V_{n+1}(x)\neq V_{n+1}(x').
\]
By the definition of the partitions,
\begin{equation}\label{eq:source-first-split}
    r_{n+1}<d_X(x,x')\le  r_n.
\end{equation}

The balls $\mathcal B_{v_{n+1}(x)}$ and
$\mathcal B_{v_{n+1}(x')}$ correspond to distinct children of the
same vertex. Therefore, their separation gives
\[
    d_Y\bigl(f(x),f(x')\bigr)>2r_{n+1}.
\]
On the other hand, both $f(x)$ and $f(x')$ belong to the common
parent ball $\mathcal B_{v_n(x)}$, and hence
\[
    d_Y\bigl(f(x),f(x')\bigr)\le  2r_n.
\]
This implies,
\begin{equation}\label{eq:target-first-split}
    2r_{n+1}
    <
    d_Y\bigl(f(x),f(x')\bigr)
    \le 
    2r_n.
\end{equation}

Since $r_n=\delta^n$, equations
\eqref{eq:source-first-split} and
\eqref{eq:target-first-split} imply
\[
    2\delta\,d_X(x,x')
    <
    d_Y\bigl(f(x),f(x')\bigr)
    <
    \frac{2}{\delta}\,d_X(x,x').
\]
Thus $f$ is a bi-Lipschitz embedding. \qed

\bibliographystyle{acm}
\bibliography{references}
\end{document}